\documentclass[12pt]{article}
\usepackage{epsf}
\usepackage{amsbsy,amsmath}
\usepackage{mathtools}
\usepackage{amsfonts}
\usepackage{amssymb}
\usepackage{enumitem}
\usepackage{eucal}
\usepackage{enumitem}
\usepackage{graphics,mathrsfs}
\usepackage{amsthm}
\usepackage{secdot}

\newtheorem{theorem}{Theorem}[section]
\newtheorem{lemma}[theorem]{Lemma}

\theoremstyle{definition}
\newtheorem{definition}[theorem]{Definition}

\theoremstyle{remark}
\newtheorem{remark}[theorem]{Remark}
\numberwithin{equation}{section}

\numberwithin{equation}{section}

\newsavebox{\savepar}

\begin{document}
	
	\title{Existence and multiplicity of solutions to a $p-q$ Laplacian system with a concave and singular nonlinearities}
	\author{Debajyoti Choudhuri$^\ddagger$, Kamel Saoudi$^{\dagger,}$\footnote{Corresponding
			author: kasaoudi@gmail.com} \& Kratou Mouna$^{\dagger,}$\\
		\small{$^\dagger$Basic and Applied Scientific Research Center, Imam Abdulrahman Bin Faisal University,}\\ \small{P.O. Box 1982, 31441, Dammam, Saudi Arabia,}\\
		\small{$^\ddagger$Department of Mathematics, National Institute of Technology Rourkela, India}\\
		\small{Emails: dc.iit12@gmail.com \& mmkratou@iau.edu.sa}}
	\date{}
	\maketitle
\begin{abstract}
\noindent In this paper we study the existence of multiple nontrivial positive weak solutions to the following system of problems.
\begin{align*}
\begin{split}
-\Delta_{p}u-\Delta_q u &= \lambda f(x)|u|^{r-2}u+\nu\frac{1-\alpha}{2-\alpha-\beta}h(x) |u|^{-\alpha}|v|^{1-\beta}\,\,\mbox{in}\,\,\Omega,\\
-\Delta_{p}v-\Delta_q v &= \mu  g(x)|v|^{r-2}v+\nu\frac{1-\beta}{2-\alpha-\beta}h(x) |u|^{1-\alpha}|v|^{-\beta}\,\,\mbox{in}\,\,\Omega,\\
u,v&>0\,\,\mbox{in}\,\,\Omega,\\
u= v &= 0\,\, \mbox{on}\,\, \partial\Omega
\end{split}
\end{align*}
where (C):~$0<\alpha<1,\;0<\beta<1,$ $2-\alpha-\beta<q<\frac{N(p-1)}{N-p}<p<r<p^*$, with $p^*=\frac{Np}{N-p}$.
We will guarantee the existence of a solution in the Nehari manifold. Further by using the Lusternik-Schnirelman category we will prove the existence of at least $\text{cat}(\Omega)+1$ number of solutions.
\begin{flushleft}
\end{flushleft}
\end{abstract}
\section{Introduction}
As mentioned in the abstract we will attempt the following problem.
\begin{align}\label{main}
\begin{split}
-\Delta_{p}u-\Delta_q u &= \lambda f(x)|u|^{r-2}u+\nu\frac{1-\alpha}{2-\alpha-\beta}h(x) |u|^{-\alpha}|v|^{1-\beta}\,\,\mbox{in}\,\,\Omega,\\
-\Delta_{p}v-\Delta_q v &= \mu  g(x)|v|^{r-2}v+\nu\frac{1-\beta}{2-\alpha-\beta}h(x) |u|^{1-\alpha}|v|^{-\beta}\,\,\mbox{in}\,\,\Omega,\\
u,v&>0\,\,\mbox{in}\,\,\Omega,\\
u= v &= 0\,\, \mbox{on}\,\, \partial\Omega
\end{split}
\end{align}
where, $\lambda, \mu, \nu>0$, $0<\alpha,\beta<1$. The functions $f,g,h\geq 0$ and are in $L^{\infty}(\Omega)$. The operator $(-\Delta_{s})$ acting on a function say $U$ is the $s$-Laplacian operator which is defined as
\begin{align*}
-\Delta_sU(x)&=-\nabla\cdot(|\nabla U|^{s-2}\nabla U)
\end{align*}
for all $s\in[1,\infty)$. We will be assuming that 
$p<N$, $1<r<q<\frac{N(p-1)}{N-1}<p<p^*$ throughout the article.
Off-late, a huge attention has been given to elliptic problems involving two Laplacian operators viz.
\begin{align*}
\begin{split}
(-\Delta_{p})u-(-\Delta_q)u &= \lambda |u|^{r-2}u+|u|^{p^*-2}u\,\,\mbox{in}\,\,\Omega,\\
u&= 0\,\, \mbox{in}\,\, \partial\Omega.
\end{split}
\end{align*}
The problem draws its motivation from the fundamental reaction-diffusion equation \begin{eqnarray}\label{citeprob1}\frac{\partial}{\partial t}u&=&\nabla\cdot[H(u)\nabla u]+c(x,u).\end{eqnarray} where $H(u)=|\nabla u|^{p-2}+|\nabla u|^{q-2}$. The problem is important owing to its manifold applications in Physics and other applied sciences such as in biophysics to model the cells, chemical reaction design, plasma physics, drug delivery mechanism to name a few. The reaction term has a polynomial form with respect to $u$. In the recent years the problem $$-\nabla\cdot[H(u)\nabla u]=c(x,u)$$ has been studied in \cite{azvre,benci,sidi,wu,li1,li3}. One may refer to Yin and Yang \cite{yin3} who studied the problem in \eqref{citeprob1} when $p^2<N$, $1<q<p<r<p*$. The authors proved the existence of $\text{cat}(\Omega)$ number of positive solutions using simple variational techniques. When $p=q$, $r=2$ the problem \eqref{citeprob1} reduces to the well-known {\it Brezis-Nirenberg problem} which has been further studied for the case of critical growth in bounded and unbounded domains by many researchers (Refer \cite{alv2,amb,cerami,rey}) and the references therein. A common issue which intrigued the researchers was to figure out a way to overcome the lack of compactness in the continuous embedding $W_0^{1,p}(\Omega)\hookrightarrow L^{p^*}(\Omega)$. Two noteworthy contributions can be found in \cite{soni,repovs}. \\
Meanwhile, the elliptic systems have also gained much attention, especially for the system 
\begin{align}\label{citeprob2}
\begin{split}
-(\Delta_{p})u&= \lambda |u|^{r-2}u+\frac{2a}{a+b}|u|^{a-2}u|v|^b \,\,\mbox{in}\,\,\Omega,\\
-(\Delta_{p})v&= \mu |v|^{r-2}v+\frac{2b}{a+b}|u|^{a}|v|^{b-2}u \,\,\mbox{in}\,\,\Omega,\\
u= v &= 0\,\, \mbox{in}\,\, \partial\Omega
\end{split}
\end{align}
where $a+b=p^*$. Ding and Xiao \cite{xiao1} studied \eqref{citeprob2} with the $p-$superlinear perturbation of $2\leq p\leq r<p^*$ an extension of which can be found in Yin \cite{yin1}. Both the works in \cite{xiao1} and \cite{yin1} have obtained the existence of $\text{cat}(\Omega)$ number of solutions using the Lusternik-Scnirelman category. For the sublinear perturbation, Hsu \cite{hsu1} obtained the existence of two positive solutions for the problem \eqref{citeprob2}. Few years back, Fan \cite{fan1} studied the problem \eqref{citeprob2} for $p=2$ and $1<r<p$. Using the Nehari manifold and the Lusternik-Schnirelman category the author has proved the admittance of at least $\text{cat}(\Omega)+1$ positive solutions. Motivated from the work of Li, Yang \cite{li2} we extend the results of the above problem with local operators and added singular nonlinearities. As far as we know there has not been any contribution in this direction whatsoever and is entirely novel. We now state the main result of this work.
\begin{theorem}\label{mainresult}
Assume the condition $(C)$ holds. Then there exists $\Lambda^*>0$ such that if $\in(0,\Lambda^*)$, problem \eqref{main} admits at least $\text{cat}(\Omega)+1$ number of distinct solutions.
\end{theorem}
\section{Preliminaries}
Let $\Omega\subset\mathbb{R}^N$, then the space $(W_0^{1,p}(\Omega), \|.\|_p)$ is defined by
\begin{eqnarray}
W_0^{1,p}(\Omega)&=&\{u:Du\in L^{p}(\Omega), u|_{\partial\Omega}=0\}\nonumber
\end{eqnarray}
equipped with the norm 
\begin{eqnarray}
\|u\|_p&=&\left(\int_{\Omega}|\nabla u|^p\right)^{\frac{1}{p}}.\nonumber
\end{eqnarray}
We will refer to $|u|_{p}$ as the $L^p$-norm of $u$ and is defined as $(\int_{\Omega}|u|^pdx)^{\frac{1}{p}}$. We further define the space 
Clearly, $X=W_0^{1,p}(\Omega)\times W_0^{1,p}(\Omega)$ is a Banach space. We define the norm of any member of $X$ as $$\|(u,v)\|_p=(\|u\|_p^p+\|v\|_p^p)^{\frac{1}{p}}.$$
The best Sobolev constant is defined as 
\begin{equation}\label{sobolev const}
S=\underset{u\in W_0^{1,p}(\Omega)\setminus\{0\}}{\inf}\cfrac{\|u\|_p^p}{\left(\int_\Omega|u|^{p^*}dx\right)^{\frac{p}{p_s^*}}}.
\end{equation}
and further define 
\begin{equation}\label{sobolev const2}
S_{a,b}=\underset{(u,v)\in X\setminus\{(0,0)\}}{\inf}\cfrac{\|(u,v)\|_p^p}{(\int_{\Omega}|u|^{p^*}+|v|^{p^*}dx)^{\frac{p}{p^*}}}.
\end{equation}
\noindent Also, we will denote $M=\|h\|_{\infty}$, $M'=\max\{\|f\|_{\infty},\|g\|_{\infty}\}$, where $\|\cdot\|_{\infty}$ denotes the essential supremum norm (or more commonly the $L^{\infty}$-norm) of a function. 
We now define the associated energy functional to the problem \eqref{main} which is as follows.
\begin{eqnarray}\label{energyfnal}
I_{\alpha,\beta}(u,v)&=&\frac{1}{p}\|(u,v)\|_p^p+\frac{1}{q}\|(u,v)\|_q^q-\frac{1}{r}\int_{\Omega}(\lambda f(x) u^r+\mu g(x)v^r)dx\nonumber\\
& &-\frac{\nu}{2-\alpha-\beta}\int_{\Omega}h(x)u^{1-\alpha}v^{1-\beta}dx.\nonumber
\end{eqnarray}
A function $(u,v)\in X$ is a weak solution to the problem (\ref{main}), if 
\begin{align*}
&(i)~ u,v>0, ~u^{-\alpha}\phi_1, v^{-\beta}\phi_2\in L^1(\Omega)~\text{and}\\
&(ii)~ \int_{\Omega}(|\nabla u|^{p-2}\nabla u\cdot\nabla \phi_1+|\nabla v|^{p-2}\nabla v\cdot\nabla \phi_2)dx+\int_{\Omega}(|\nabla u|^{q-2}\nabla u\cdot\nabla \phi_1+|\nabla v|^{q-2}\nabla v\cdot\nabla \phi_2)dx\nonumber\\
&~~~~-\int_{\Omega}(\lambda f(x) u^{r-1}\phi_1+\mu g(x) v^{r-1}\phi_2)dx-\nu\frac{1-\alpha}{2-\alpha-\beta}\int_{\Omega}h(x)u^{-\alpha}v^{1-\beta}\phi_1dx\nonumber\\
&~~~~-\nu\frac{1-\beta}{2-\alpha-\beta}\int_{\Omega}h(x)u^{1-\alpha}v^{-\beta}\phi_2dx=0
\end{align*}
for each $\phi_2,\phi_2\in X$. Note that the nontrivial critical points of the functional $I_{\alpha,\beta}$ are the positive weak solutions of the problem \eqref{main}. Note that the functional $I_{\alpha,\beta}$ is not a $C^1$-functional and hence the classical variational methods are not applicable. One can easily verify that the energy functional $I_{\alpha,\beta}$ is not bounded below in $X$. However, we will show that $I_{\alpha,\beta}$ is bounded below on a {\it Nehari manifold} and we will extract solutions by minimizing the functional on suitable subsets.
\noindent We further define the Nehari manifold as follows.
$$\mathcal{N}_{\alpha,\beta}=\{(u,v)\in Z\setminus(0,0), u,v>0:\langle I'_{\alpha,\beta}(u,v),(u,v)=0\rangle\}.$$
It is not difficult to see that a pair $(u,v)\in\mathcal{N}_{\alpha,\beta}$ if and only if 
$$\|(u,v)\|_p^p+\|(u,v)\|_q^q-\int_{\Omega}(\lambda f(x)u^r+\mu g(x)v^r)dx-\nu\int_{\Omega}h(x)u^{1-\alpha}v^{1-\beta}dx=0.$$
Furthermore, it is customary to see, as for any problem which has an involvement of a Nehari manifold, that 
\begin{eqnarray}
I_{\alpha,\beta}(u,v)&=&\left(\frac{1}{p}-\frac{1}{r}\right)\|(u,v)\|_p^p+\left(\frac{1}{q}-\frac{1}{r}\right)\|(u,v)\|_q^q+\nu\left(\frac{1}{r}-\frac{1}{2-\alpha-\beta}\right)\int_{\Omega}h(x)u^{1-\alpha}v^{1-\beta}dx.\nonumber\\
&\geq&\left(\frac{1}{p}-\frac{1}{r}\right)\left(\|(u,v)\|_p^p+\|(u,v)\|_q^q\right)+\nu\left(\frac{1}{r}-\frac{1}{2-\alpha-\beta}\right)\int_{\Omega}h(x)u^{1-\alpha}v^{1-\beta}dx.\nonumber\\
&\geq&\left(\frac{1}{p}-\frac{1}{r}\right)\|(u,v)\|_p^p+\nu\left(\frac{1}{r}-\frac{1}{2-\alpha-\beta}\right)\int_{\Omega}h(x)u^{1-\alpha}v^{1-\beta}dx\nonumber\\
&\geq&\left(\frac{1}{p}-\frac{1}{r}\right)\|(u,v)\|_p^p-\nu\left(\frac{1}{2-\alpha-\beta}-\frac{1}{r}\right)\|(u,v)\|_p^{2-\alpha-\beta}.\nonumber
\end{eqnarray}
Since $2-\alpha-\beta<p$, therefore $I_{\alpha,\beta}$ is coercive and bounded below on $\mathcal{N}_{\alpha,\beta}$. Therefore the functional is coercive and is bounded below in $\mathcal{N}_{\alpha,\beta}$. In fact $I_{\alpha,\beta}(u,v)\geq 0$ for sufficiently small $\nu>0$ and for all $(u,v)\in\mathcal{N}_{\alpha,\beta}$. We define for $t\geq 0$ the fiber maps
\begin{eqnarray}
\Phi_{\alpha,\beta}(t)=I_{\alpha,\beta}(tu,tv)&=&\frac{t^p}{p}\|(u,v)\|_p^p+\frac{t^q}{q}\|(u,v)\|_q^q\nonumber\\
& &-\frac{t^r}{r}\int_{\Omega}(\lambda f(x) u^r+\mu g(x)v^r)dx-\nu \frac{t^{2-\alpha-\beta}}{2-\alpha-\beta}\int_{\Omega}h(x)u^{1-\alpha}v^{1-\beta}dx.\nonumber
\end{eqnarray}
Then
\begin{eqnarray}
\Phi_{\alpha,\beta}'(t)&=&t^{p-1}\|(u,v)\|_p^p+t^{q-1}\|(u,v)\|_q^q-t^{r-1}\int_{\Omega}(\lambda f(x) u^r+\mu g(x)v^r)dx\nonumber\\
& &-\nu t^{1-\alpha-\beta}\int_{\Omega}h(x)u^{1-\alpha}v^{1-\beta}dx\nonumber
\end{eqnarray}
and 
\begin{eqnarray}
\Phi_{\alpha,\beta}''(t)&=&(p-1)t^{p-2}\|(u,v)\|_p^p+(q-1)t^{q-2}\|(u,v)\|_q^q-(r-1)t^{r-2}\int_{\Omega}(\lambda f(x) u^r+\mu g(x)v^r)dx\nonumber\\
& &-\nu(1-\alpha-\beta) t^{-\alpha-\beta}\int_{\Omega}h(x)u^{1-\alpha}v^{1-\beta}dx.\nonumber
\end{eqnarray}
A simple observation shows that $(u,v)\in\mathcal{N}_{\alpha,\beta}$ if and only if $\Phi_{\alpha,\beta}'(1)=0$. Furthermore, in general we have that $(u,v)\in\mathcal{N}_{\alpha,\beta}$ if and only if $\Phi_{\alpha,\beta}'(t)=0$. Therefore for $(u,v)\in\mathcal{N}_{\alpha,\beta}$ we have
\begin{eqnarray}
\Phi_{\alpha,\beta}''(1)&=&(p-1)\|(u,v)\|_p^p+(q-1)\|(u,v)\|_q^q-(r-1)\int_{\Omega}(\lambda f(x) u^r+\mu g(x)v^r)dx\nonumber\\
& &-\nu(1-\alpha-\beta) \int_{\Omega}h(x)u^{1-\alpha}v^{1-\beta}dx\nonumber\\
&=&(p-r)\|(u,v)\|_p^p+(q-r)\|(u,v)\|_q^q+\nu(r+\alpha+\beta-2)\int_{\Omega}h(x)u^{1-\alpha}v^{1-\beta}dx\nonumber\\
&=&(p+\alpha+\beta-2)\|(u,v)\|_p^p+(q+\alpha+\beta-2)\|(u,v)\|_q^q\nonumber\\
& &+(2-\alpha-\beta-r)\int_{\Omega}(\lambda f(x) u^r+\mu g(x)v^r)dx.\nonumber
\end{eqnarray}
Therefore we split the Nehari manifold into three parts, namely
\begin{eqnarray}
\mathcal{N}^{+}_{\alpha,\beta}&=&\{(u,v)\in\mathcal{N}_{\alpha,\beta}:\Phi_{\alpha,\beta}''(1)>0\},\nonumber\\
\mathcal{N}^{-}_{\alpha,\beta}&=&\{(u,v)\in\mathcal{N}_{\alpha,\beta}:\Phi_{\alpha,\beta}''(1)<0\},\nonumber\\
\mathcal{N}^0_{\alpha,\beta}&=&\{(u,v)\in\mathcal{N}_{\alpha,\beta}:\Phi_{\alpha,\beta}''(1)=0\}\nonumber
\end{eqnarray}
which corresponds to the collection of local minima, maxima and points of inflection respectively.
We now prove a lemma which falls back on the proof due to Hsu \cite{hsu1} (refer Theorem 2.2).
\begin{lemma}\label{empty}
For $(u,v)\in\mathcal{N}_{\alpha,\beta}$, there exists a positive constant $A_0$, that depends on $p, S, N, \alpha, \beta, |\Omega|$ such that $I_{\alpha,\beta}(u,v)\geq-A_0\left[\left(\frac{1-\alpha}{2-\alpha-\beta}\right)^{\frac{p+\alpha+\beta-2}{p}}+\left(\frac{1-\beta}{2-\alpha-\beta}\right)^{\frac{p+\alpha+\beta-2}{p}}\right]$.
\end{lemma}
\begin{proof}
We use
\begin{eqnarray}\label{ineq1}
I_{\alpha,\beta}(u,v)&\geq&\left(\frac{1}{p}-\frac{1}{r}\right)\left(\|(u,v)\|_p^p\right)+\nu\left(\frac{1}{r}-\frac{1}{2-\alpha-\beta}\right)\int_{\Omega}h(x)u^{1-\alpha}v^{1-\beta}dx.\nonumber\\
\end{eqnarray}
By the H\"{o}lder inequality, the Young's inequality, and the Sobolev embedding theorem to \eqref{ineq1}, we have
\begin{eqnarray}
I_{\alpha,\beta}(u,v)&\geq&\left(\frac{1}{p}-\frac{1}{r}\right)\left(\|(u,v)\|_p^p\right)-\nu\left(\frac{1}{2-\alpha-\beta}-\frac{1}{r}\right)\int_{\Omega}h(x)u^{1-\alpha}v^{1-\beta}dx\nonumber\\
&\geq&\left(\frac{1}{p}-\frac{1}{r}\right)\left(\|(u,v)\|_p^p\right)\nonumber\\
& &-\nu M|\Omega|^{1-\frac{2-\alpha-\beta}{p^*}}\nonumber\\
& &\times\left(\frac{1}{2-\alpha-\beta}-\frac{1}{r}\right)\int_{\Omega}\left(\frac{1-\alpha}{2-\alpha-\beta}|u|_{p^*}^{2-\alpha-\beta}+\frac{1-\beta}{2-\alpha-\beta}|v|_{p^*}^{2-\alpha-\beta}\right)dx\nonumber\\
&\geq&\left(\frac{1}{p}-\frac{1}{r}\right)\left(\|(u,v)\|_p^p\right)\nonumber\\
& &-\nu M|\Omega|^{1-\frac{2-\alpha-\beta}{p^*}}S^{\frac{\alpha+\beta-2}{p}}\nonumber\\
& &\times\left(\frac{1}{2-\alpha-\beta}-\frac{1}{r}\right)\int_{\Omega}\left(\frac{1-\alpha}{2-\alpha-\beta}|\nabla u|_{p}^{2-\alpha-\beta}+\frac{1-\beta}{2-\alpha-\beta}|\nabla  v|_{p}^{2-\alpha-\beta}\right)dx\nonumber\\
&\geq&-\nu A_0(p, S, N, \alpha, \beta, |\Omega|)\left[\left(\frac{1-\alpha}{2-\alpha-\beta}\right)^{\frac{p}{p+\alpha+\beta-2}}+\left(\frac{1-\beta}{2-\alpha-\beta}\right)^{\frac{p}{p+\alpha+\beta-2}}\right].\nonumber\\
\end{eqnarray}
\end{proof}
\begin{lemma}\label{LBofI2}
There exists $\Lambda^*>0$ such that if~ $\nu\left[\left(\frac{1-\alpha}{2-\alpha-\beta}\right)^{\frac{p}{p+\alpha+\beta-2}}+\left(\frac{1-\beta}{2-\alpha-\beta}\right)^{\frac{p}{p+\alpha+\beta-2}}\right]\in (0,\Lambda^*)$, then $\mathcal{N}_{\alpha,\beta}^0=\phi$.	
\end{lemma}
\begin{proof}
Let us choose $$\Lambda^*=\left((p-2+\alpha+\beta)\frac{1}{M'(\lambda+\mu)}\right)^{\frac{p}{r-p}}\frac{(r-p)S_{\alpha,\beta}^{\frac{rp}{N(r-p)}+\frac{2-\alpha-\beta}{p}}}{\nu M(r-2+\alpha+\beta)^{\frac{r}{r-p}}|\Omega|^{1-\frac{2-\alpha-\beta}{p^*}}}.$$ 
The proof follows by contradiction.
\end{proof}
\noindent From the lemma \eqref{LBofI2}, we have that if $\nu\left[\left(\frac{1-\alpha}{2-\alpha-\beta}\right)^{\frac{p}{p+\alpha+\beta-2}}+\left(\frac{1-\beta}{2-\alpha-\beta}\right)^{\frac{p}{p+\alpha+\beta-2}}\right]\in (0,\Lambda^*)$, then  $\mathcal{N}_{\alpha,\beta}=\mathcal{N}_{\alpha,\beta}^{+}\bigcup\mathcal{N}_{\alpha,\beta}^{-}$. We can define $i^+=\inf_{(u,v)\in\mathcal{N}_{\alpha,\beta}^{+}}I_{\alpha,\beta}$ and $i^-=\inf_{(u,v)\in\mathcal{N}_{\alpha,\beta}^{-}}I_{\alpha,\beta}$ since the functional $I_{\alpha,\beta}$ is bounded below in $\mathcal{N}_{\alpha,\beta}$.\\
\begin{remark}
We will denote the norm convergence by $\rightarrow$, the weak convergence by $\rightharpoonup$ and $\Lambda$ as any small parameter we will encounter or any cumbersome representation in short form.
\end{remark}
\begin{lemma}\label{LBofI3}
There exists $\Lambda^*>0$ such that if $\nu\left[\left(\frac{1-\alpha}{2-\alpha-\beta}\right)^{\frac{p}{p+\alpha+\beta-2}}+\left(\frac{1-\beta}{2-\alpha-\beta}\right)^{\frac{p}{p+\alpha+\beta-2}}\right]\in (0,\Lambda^*)$, then 
\begin{enumerate}
\item $i^{+}<0$,
\item $i^-\geq D_0$ for some $D_0>0$. 
\end{enumerate}
\end{lemma}
\begin{proof}
\begin{enumerate}
	\item Let $(u,v)\in\mathcal{N}_{\alpha,\beta}^+\subset\mathcal{N}_{\alpha,\beta}$. Then we have
\begin{eqnarray}
0<(r-p)\|(u,v)\|_p^p+(r-q)\|(u,v)\|_q^q&<&\nu(r+\alpha+\beta-2)\int_{\Omega}h(x)u^{1-\alpha}v^{1-\beta}dx\nonumber\\
\end{eqnarray}
Further, 
\begin{eqnarray}
I_{\alpha,\beta}(u,v)&=&\left(\frac{1}{p}-\frac{1}{r}\right)\|(u,v)\|_p^p+\left(\frac{1}{q}-\frac{1}{r}\right)\|(u,v)\|_q^q\nonumber\\
& &+\nu\left(\frac{1}{r}-\frac{1}{2-\alpha-\beta}\right)\int_{\Omega}h(x)u^{1-\alpha}v^{1-\beta}dx.\nonumber\\
&< &\left(\frac{1}{p}-\frac{1}{r}\right)\|(u,v)\|_p^p+\left(\frac{1}{q}-\frac{1}{r}\right)\|(u,v)\|_q^q\nonumber\\
& &-\frac{(r-p)}{r(2-\alpha-\beta)}\|(u,v)\|_p^p-\frac{(r-q)}{r(2-\alpha-\beta)}\|(u,v)\|_q^q\nonumber\\
&= &\frac{(r-p)}{r}\left(\frac{1}{p}-\frac{1}{2-\alpha-\beta}\right)\|(u,v)\|_p^p+\frac{(r-p)}{r}\left(\frac{1}{q}-\frac{1}{2-\alpha-\beta}\right)\|(u,v)\|_q^q\nonumber\\
& &<0.\nonumber
\end{eqnarray}
Therefore, $i^+=\inf_{(u,v)\in\mathcal{N}_{\alpha,\beta}^+}I_{\alpha,\beta}(u,v)<0$.\\
\item Likewise, let us choose $(u,v)\in\mathcal{N}_{\alpha,\beta}^-$. We again appeal to the following inequality 
\begin{eqnarray}
(p+\alpha+\beta-2)\|(u,v)\|_p^p&<&(p+\alpha+\beta-2)\|(u,v)\|_p^p+(q+\alpha+\beta-2)\|(u,v)\|_q^q\nonumber\\
&<&(r+\alpha+\beta-2)\int_{\Omega}(\lambda f(x)u^r+\mu g(x)v^r)dx\nonumber\\
&\leq&(r+\alpha+\beta-2)CM'(\lambda^{\frac{r}{r-p}}+\mu^{\frac{r}{r-p}})\|(u,v)\|_p^r.\nonumber\\
\end{eqnarray}
by virtue of the fact that $(u,v)\in\mathcal{N}_{\alpha,\beta}$. Therefore
$$\|(u,v)\|_p\geq \left[\left(\frac{p+\alpha+\beta-2}{r+\alpha+\beta-2}\right)\frac{1}{CM'(\lambda^{\frac{r}{r-p}}+\mu^{\frac{r}{r-p}})}\right]^{\frac{1}{r-p}}.$$
We will call this cumbersome looking constant as $\Lambda$. 
Therefore on proceeding further we have 
\begin{eqnarray}
I_{\alpha,\beta}(u,v)&=&\left(\frac{1}{p}-\frac{1}{r}\right)\|(u,v)\|_p^p+\left(\frac{1}{q}-\frac{1}{r}\right)\|(u,v)\|_q^q\nonumber\\
& &+\nu\left(\frac{1}{r}-\frac{1}{2-\alpha-\beta}\right)\int_{\Omega}h(x)u^{1-\alpha}v^{1-\beta}dx\nonumber\\
& &\geq \left(\frac{1}{p}-\frac{1}{r}\right)\|(u,v)\|_p^p\nonumber\\
& &-\nu M|\Omega|^{1-\frac{2-\alpha-\beta}{p^*}}S^{\frac{\alpha+\beta-2}{p}}\nonumber\\
& &\times\left(\frac{1}{2-\alpha-\beta}-\frac{1}{r}\right)\int_{\Omega}\left(\frac{1-\alpha}{2-\alpha-\beta}|\nabla u|_{p}^{2-\alpha-\beta}+\frac{1-\beta}{2-\alpha-\beta}|\nabla  v|_{p}^{2-\alpha-\beta}\right)dx\nonumber\\
&\geq&\left(\frac{1}{p}-\frac{1}{r}\right)\|(u,v)\|_p^p\nonumber\\
& &-A_0(p, s, N, \alpha, \beta, |\Omega|)\left[\left(\frac{1-\alpha}{2-\alpha-\beta}\right)^{\frac{p}{p+\alpha+\beta-2}}+\left(\frac{1-\beta}{2-\alpha-\beta}\right)^{\frac{p}{p+\alpha+\beta-2}}\right]\|(u,v)\|_p^{2-\alpha-\beta}\nonumber\\
&=&\|(u,v)\|_p^{2-\alpha-\beta}\left[\left(\frac{1}{p}-\frac{1}{r}\right)\|(u,v)\|_p^{p+\alpha+\beta-2}\right.\nonumber\\
& &\left.-A_0(p, s, N, \alpha, \beta, |\Omega|)\left\{\left(\frac{1-\alpha}{2-\alpha-\beta}\right)^{\frac{p}{p+\alpha+\beta-2}}+\left(\frac{1-\beta}{2-\alpha-\beta}\right)^{\frac{p}{p+\alpha+\beta-2}}\right\}\right].\nonumber\\
& &\Lambda^{2-\alpha-\beta}\left[\left(\frac{1}{p}-\frac{1}{r}\right)\Lambda^{p+\alpha+\beta-2}\right.\nonumber\\
& &\left.-A_0(p, s, N, \alpha, \beta, |\Omega|)\left\{\left(\frac{1-\alpha}{2-\alpha-\beta}\right)^{\frac{p}{p+\alpha+\beta-2}}+\left(\frac{1-\beta}{2-\alpha-\beta}\right)^{\frac{p}{p+\alpha+\beta-2}}\right\}\right].\nonumber
\end{eqnarray}
Then for a sufficiently small $\Lambda^*>0$ and $D_0>0$ such that $\left(\frac{1-\alpha}{2-\alpha-\beta}\right)^{\frac{p+\alpha+\beta-2}{p}}+\left(\frac{1-\beta}{2-\alpha-\beta}\right)^{\frac{p+\alpha+\beta-2}{p}}\in(0,\Lambda^*)$, we have $i^-\geq D_0>0$.
\end{enumerate}
\end{proof}
\begin{remark}
For a better understanding of the Nehari manifold and the fiber maps, we define the function 
$$F_{u,v}(t)=t^{p-r}\|(u,v)\|_p^p+t^{q-r}\|(u,v)\|_q^q-\nu t^{2-\alpha-\beta-r}\int_{\Omega}h(x)u^{1-\alpha}v^{1-\beta}dx.$$
Then 
$$\Phi'(t)=t^{r-1}[F_{u,v}(t)-\int_{\Omega}(\lambda f(x)u^r+\beta g(x)v^r)dx].$$	
\end{remark}
observe that $\lim_{t\rightarrow\infty}F_{u,v}(t)=0$ and $\lim_{t\rightarrow 0^+}F_{u,v}(t)=-\infty$. Further, 
\begin{eqnarray}
F'_{u,v}(t)&=&(p-r)t^{p-r-1}\|(u,v)\|_p^p+(q-r)t^{q-r-1}\|(u,v)\|_q^q\nonumber\\
& &-\nu(2-\alpha-\beta-r)t^{1-\alpha-\beta-r}\int_{\Omega}h(x)u^{1-\alpha}v^{1-\beta}dx\nonumber\\
&=&t^{1-\alpha-\beta-r}[(p-r)t^{p+\alpha+\beta}\|(u,v)\|_p^p+(q-r)t^{q+\alpha+\beta}\|(u,v)\|_q^q\nonumber\\
& &-\nu(2-\alpha-\beta-r)\int_{\Omega}h(x)u^{1-\alpha}v^{1-\beta}dx].\nonumber
\end{eqnarray}
Let $$\psi_{u,v}(t)=(p-r)t^{p+\alpha+\beta}\|(u,v)\|_p^p+(q-r)t^{q+\alpha+\beta}\|(u,v)\|_q^q
-\nu(2-\alpha-\beta-r)\int_{\Omega}h(x)u^{1-\alpha}v^{1-\beta}dx.$$
We also have $\lim_{t\rightarrow 0^+}\psi_{u,v}(t)=\nu(r+\alpha+\beta-2)\int_{\Omega}h(x)u^{1-\alpha}v^{1-\beta}dx$, $\lim_{t\rightarrow\infty}\psi_{u,v}(t)=-\infty$ and
$$\psi'_{u,v}(t)=(p-r)(p+\alpha+\beta)t^{p+\alpha+\beta-1}\|(u,v)\|_p^p+(q-r)(q+\alpha+\beta)t^{q+\alpha+\beta-1}\|(u,v)\|_q^q<0.$$
Thus, for each $(u,v)\in X$ with $\int{\Omega}h(x)u^{1-\alpha}v^{1-\beta}dx>0$, $F_{u,v}(t)$ attains its maximum at some $t_{max}=t_{max}(u,v)$. This unique $t_{max}$ can be evaluated by solving for $t$ from the equation 
$$(r-p)t^{p+\alpha+\beta}\|(u,v)\|_p^p+(r-q)t^{q+\alpha+\beta}\|(u,v)\|_q^q
=\nu(r+\alpha+\beta-2)\int_{\Omega}h(x)u^{1-\alpha}v^{1-\beta}dx.$$
A simple calculation yields 
$$F_{u,v}(t_{max})=t_{max}^{p-r}\left(1+\frac{r-p}{r+\alpha_\beta-2}t_{max}^2\right)\|(u,v)\|_{p}^p+t_{max}^{q-r}\left(1+\frac{r-q}{r+\alpha_\beta-2}t_{max}^2\right)\|(u,v)\|_{q}^q>0.$$
Thus for $t\in(0,t_{max})$ we have $F'_{u,v}(t)>0$ and $F'_{u,v}(t)<0$.\\
We now have the following lemma as a consequence.
\begin{lemma}
For every $(u,v)\in X\setminus\{(0,0)\}$ there exists a unique $0<t^+<t_{max}$ such that $(t^+u,t^+v)\in\mathcal{N}_{\alpha,\beta}^+$ and 
$$I_{\alpha,\beta}(t^+u,t^+v)=\inf_{t\geq 0}I_{\alpha,\beta}(tu,tv).$$
Furthermore, if $\int_{\Omega}(\lambda f(x)u^r+\mu g(x)v^r)dx>0$ then there exists unique $0<t^+<t_{max}<t^-$ such that $(t^+u,t^+v)\in\mathcal{N}_{\alpha,\beta}^+$, $(t^-u,t^-v)\in\mathcal{N}_{\alpha,\beta}^-$ and
$$I_{\alpha,\beta}(t^+u,t^+v)=\inf_{0\leq t\leq t_{max}}I_{\alpha,\beta}(tu,tv),~~~ I_{\alpha,\beta}(t^-u,t^-v)=\sup_{t\geq 0}I_{\alpha,\beta}(tu,tv).$$
\end{lemma}
\begin{proof}
We only prove the case when $\int_{\Omega}(\lambda f(x)u^r+\mu g(x)v^r)dx>0$. Thus the equation $F_{u,v}(t)=\int_{\Omega}(\lambda f(x)u^r+\beta g(x)v^r)dx$ has only two solutions namely, $0<t^+<t_{max}<t^-$ such that $I'_{\alpha,\beta}(t^+)>0$ and $I'_{\alpha,\beta}(t^-)<0$. Since 
$$\Phi''(t^+)=(t^+)^{r-1}[F_{u,v}(t^+)-\int_{\Omega}(\lambda f(x)u^r+\mu g(x)v^r)dx]>0$$	and
$$\Phi''(t^-)=(t^-)^{r-1}[F_{u,v}(t^-)-\int_{\Omega}(\lambda f(x)u^r+\mu g(x)v^r)dx]<0,$$ therefore $(t^+u,t^+v)\in \mathcal{N}_{\alpha,\beta}^+$ and $(t^-u,t^-v)\in \mathcal{N}_{\alpha,\beta}^-$. Thus $\Phi(t)$ decreases in $(0,t^+)$, increases in $(t^+,t^-)$ and decreases in $(t^-,\infty)$. Hence the lemma.
\end{proof}
\noindent We now define the palais-Smale sequence ((PS)-sequence), (PS)-condition and (PS)-value in $X$ for $I_{\alpha,\beta}$ corresponding to the functional $I_{\alpha,\beta}$ which is as follows.
\begin{definition}
	Suppose for $c\in\mathbb{R}$, a sequence $\{(u_n,v_n)\}\subset X$ is a $(PS)_c$-sequence for the functional $I_{\alpha,\beta}$ if $I_{\alpha,\beta}(u_n,v_n)\rightarrow c$ and $I'_{\alpha,\beta}(u_n,v_n)\rightarrow 0$ in $X'$ as $n\rightarrow\infty$, then:
\begin{enumerate}
\item $c\in\mathbb{R}$ is a $(PS)$ value in $X$ for the functional $I_{\alpha,\beta}$ if there exists a $(PS)_c$-sequence in $X$ for $I_{\alpha,\beta}$.
\item The functional $I_{\alpha,\beta}$ satisfies the $(PS)_c$-condition in $X$ for $I_{\alpha,\beta}$ if any $(PS)_c$-sequence admits a strongly convergent subsequence in $X$.
\end{enumerate}
\end{definition}
\begin{remark}
We will sometimes denote $\lim_{n\rightarrow\infty}x_n=0$ as $x_n=o(1)$ for a sequence of real numbers $(x_n)$.
\end{remark}
\begin{lemma}\label{PSbounded}
For any $0<\alpha,\beta<1$, the functional $I_{\alpha,\beta}$ satisfies the $(PS)_c$-condition for $c\in\left(-\infty,\frac{S_{\alpha,\beta}^{\frac{r}{r-p}}}{\Lambda}-\nu A_0\left[\left(\frac{1-\alpha}{2-\alpha-\beta}\right)^{\frac{p}{p+\alpha+\beta-2}}+\left(\frac{1-\beta}{2-\alpha-\beta}\right)^{\frac{p}{p+\alpha+\beta-2}}\right]\right)$ where $\Lambda=2M'(\lambda^{\frac{r}{r-p}}+\mu^{\frac{r}{r-p}})\}^{\frac{p}{r-p}}|\Omega|^{\frac{1}{r}}$.
\end{lemma}
\begin{proof}
Suppose $\{(u_n,v_n)\}$ is a $(PS)_c$-sequence in $X$ for the functional $I_{\alpha,\beta}$ with $c\in\left(-\infty,\frac{S_{\alpha,\beta}^{\frac{r}{r-p}}}{\Lambda}-\nu A_0\left[\left(\frac{1-\alpha}{2-\alpha-\beta}\right)^{\frac{p}{p+\alpha+\beta-2}}+\left(\frac{1-\beta}{2-\alpha-\beta}\right)^{\frac{p}{p+\alpha+\beta-2}}\right]\right)$. Then
\begin{eqnarray}\label{PS1}
I_{\alpha,\beta}(u_n,v_n)=c+o(1), ~~~I'_{\alpha,\beta}(u_n,v_n)=o(1).
\end{eqnarray}
We now claim that $\{(u_n,v_n)\}$ is bounded in $X$. We prove this claim by contradiction, i.e. say $\|(u_n,v_n)\|_p\rightarrow\infty$. Let $(\tilde{u_n},\tilde{v_n})=\left(\frac{u_n}{\|(u_n,v_n)\|_p},\frac{v_n}{\|(u_n,v_n\|_p)}\right)$, then $\|(\tilde{u_n},\tilde{v_n})\|_p=1$ which implies that $(\tilde{u_n},\tilde{v_n})$ is bounded in $X$. Therefore, due to the reflexivity of the space $X$, we have upto a subsequence 
$$(\tilde{u_n},\tilde{v_n})\rightharpoonup(u_n,v_n)$$ as $n\rightarrow\infty$ in $X$. This further implies that 
$$\tilde{u_n}\rightharpoonup\tilde{u},~~~\tilde{v_n}\rightharpoonup\tilde{v}~\text{in}~W_0^{1,p}(\Omega),$$
$$\tilde{u_n}\rightarrow\tilde{u},~~~\tilde{v_n}\rightarrow\tilde{v}~\text{in}~L^{s}(\Omega),~1\leq s<p^*,$$
$$\int_{\Omega}\nu h(x)\tilde{u_n}^{1-\alpha}\tilde{v_n}^{1-\beta}dx\rightarrow\int_{\Omega}\nu h(x)u^{1-\alpha}v^{1-\beta}dx.$$
The last convergence follows from the Egoroff's theorem. From \eqref{PS1} we have
\begin{eqnarray}c+o(1)&=&\frac{1}{p}\|(u_n,v_n)\|_p^p\|(\tilde{u_n},\tilde{v}_n)\|_p^p+\frac{1}{q}\|(u_n,v_n)\|_q^q\|(\tilde{u_n},\tilde{v}_n)\|_q^q\nonumber\\
& &-\frac{1}{r}\|(u_n,v_n)\|_{p}^r\int_{\Omega}(\lambda f(x) \tilde{u}_n^r+\mu g(x)\tilde{v}_n^r)dx\nonumber\\
& &-\frac{\nu}{2-\alpha-\beta}\|(u_n,v_n)\|_{p}^{2-\alpha-\beta}\int_{\Omega}h(x)\tilde{u}_n^{1-\alpha}\tilde{v}_n^{1-\beta}dx\nonumber\end{eqnarray}
and 
\begin{eqnarray}o(1)&=&\|(u_n,v_n)\|_p^p\|(\tilde{u}_n,\tilde{v}_n)\|_p^p+\|(u_n,v_n)\|_q^q\|(\tilde{u}_n,\tilde{v}_n)\|_q^q\nonumber\\
& &-\|(u_n,v_n)\|_p^r\int_{\Omega}(\lambda f(x) \tilde{u}_n^r+\mu g(x)\tilde{v}_n^r)dx\nonumber\\
& &-\nu\|(u_n,v_n)\|_{p}^{2-\alpha-\beta}\int_{\Omega}h(x)\tilde{u}_n^{1-\alpha}\tilde{v}_n^{1-\beta}dx.\nonumber\end{eqnarray}
Now by the assumption we made, i.e. $\|(u_n,v_n)\|_p\rightarrow\infty$, we obtain
\begin{eqnarray}o(1)&=&\frac{1}{p}\|(\tilde{u_n},\tilde{v}_n)\|_p^p+\frac{1}{q}\|(u_n,v_n)\|_p^{q-p}\|(\tilde{u_n},\tilde{v}_n)\|_q^q\nonumber\\
& &-\frac{1}{r}\|(u_n,v_n)\|_{p}^{r-p}\int_{\Omega}(\lambda f(x) \tilde{u}_n^r+\mu g(x)\tilde{v}_n^r)dx\nonumber\\
& &-\frac{\nu}{2-\alpha-\beta}\|(u_n,v_n)\|_{p}^{2-\alpha-\beta-p}\int_{\Omega}h(x)\tilde{u}_n^{1-\alpha}\tilde{v}_n^{1-\beta}dx\nonumber\end{eqnarray}
and
\begin{eqnarray}o(1)&=&\|(\tilde{u}_n,\tilde{v}_n)\|_p^p+\|(u_n,v_n)\|_p^{q-p}\|(\tilde{u}_n,\tilde{v}_n)\|_q^q\nonumber\\
& &-\|(u_n,v_n)\|_p^{r-p}\int_{\Omega}(\lambda f(x) \tilde{u}_n^r+\mu g(x)\tilde{v}_n^r)dx\nonumber\\
& &-\nu\|(u_n,v_n)\|_{p}^{2-\alpha-\beta-p}\int_{\Omega}h(x)\tilde{u}_n^{1-\alpha}\tilde{v}_n^{1-\beta}dx.\nonumber\end{eqnarray}
On using the above to equalities we get
\begin{eqnarray}
o(1)&=&\left(1-\frac{2-\alpha-\beta}{p}\right)\|(\tilde{u}_n,\tilde{v}_n)\|_{p}^p+\left(1-\frac{2-\alpha-\beta}{q}\right)\|(u_n,v_n)\|_p^{q-p}\|(\tilde{u}_n,\tilde{v}_n)\|_{q}^{q}\nonumber\\
& &+\left(\frac{2-\alpha-\beta}{r}-1\right)\|(u_n,v_n)\|_p^{r-p}\int_{\Omega}(\lambda f(x) \tilde{u}_n^r+\mu g(x)\tilde{v}_n^r)dx\nonumber
\end{eqnarray}
as $n\rightarrow\infty$. Therefore we have 
\begin{eqnarray}
\|(\tilde{u}_n,\tilde{v}_n)\|_{p}^p&=&\frac{p(p-2+\alpha+\beta)}{q(2-\alpha-\beta-q)}\|(u_n,v_n)\|_p^{q-p}\|(\tilde{u}_n,\tilde{v}_n)\|_{q}^{q}\nonumber\\
& &+\nu\frac{p(p-2+\alpha+\beta)}{r(r-2+\alpha+\beta)}\|(u_n,v_n)\|_{p}^{2-\alpha-\beta-p}\int_{\Omega}h(x)\tilde{u}_n^{1-\alpha}\tilde{v}_n^{1-\beta}dx+o(1)\nonumber
\end{eqnarray}
as $n\rightarrow\infty$, Thus we have $\|(\tilde{u}_n,\tilde{v}_n)\|_{p}^p\rightarrow\infty$ which is a contradiction to our assumption that $\|(\tilde{u}_n,\tilde{v}_n)\|_{p}=1$. Therefore, the sequence $\{(u_n,v_n)\}$ is bounded in $X$.\\
\noindent We choose a subsequence to this bounded sequence, still denoted by $\{(u_n,v_n)\}$ such that 
$$(u_n,v_n)\rightharpoonup(u,v)~\text{in}~X,$$
$$u_n\rightarrow u,~~~v_n\rightarrow v~\text{in}~L^s(\Omega),~1\leq s<p^*,$$
$$\int_{\Omega}(\lambda f(x)u_n^r+\mu g(x)v_n^r)dx\rightarrow \int_{\Omega}(\lambda f(x)u^r+\mu g(x)v^r)dx,$$
$$\nu\int_{\Omega}h(x)u_n^{1-\alpha}v_n^{1-\beta}dx\rightarrow\nu\int_{\Omega}h(x)u^{1-\alpha}v^{1-\beta}dx$$
as $n\rightarrow\infty$.\\
By the Brezis-Lieb \cite{li4} theorem we get
$$\|(u_n-u,v_n-v)\|_p^p=\|(u_n,v_n)\|_p^p-\|(u,v)\|_p^p+o(1),$$
\begin{eqnarray}
\int_{\Omega}(\lambda f(x)(u_n-u)^r+\mu g(x)(v_n-v)^r)dx&=&\int_{\Omega}(\lambda f(x)u_n^r+\mu g(x)v_n^r)dx-\nonumber\\
&& \int_{\Omega}(\lambda f(x)u^r+\mu g(x)v^r)dx+o(1)\nonumber
\end{eqnarray}
and
\begin{eqnarray}
\nu\int_{\Omega}h(x)(u_n-u)^{1-\alpha}(v_n-v)^{1-\beta}dx&=&\nu\int_{\Omega}h(x)u_n^{1-\alpha}v_n^{1-\beta}dx-\nu\int_{\Omega}h(x)u^{1-\alpha}v^{1-\beta}dx+o(1).\nonumber
\end{eqnarray}
Thus for any $(\phi_2,\phi_2)\in X$ the following holds.
$$\lim_{n\rightarrow\infty}\langle I'_{\alpha,\beta},(\phi_2,\phi_2)\rangle=\langle I'_{\alpha,\beta}(u,v),(\phi_1,\phi_2)\rangle=0.$$
In other words $(u,v)$ is a critical point of $I_{\alpha,\beta}$.
All we now need to show is that $(u_n,v_n)\rightarrow(u,v)$ in $X$. We use \eqref{PS1}, the Brezis-Lieb lemma \cite{li4} and some basic functional analysis to obtain 
\begin{eqnarray}\label{PS2}
c-I_{\alpha,\beta}+o(1)&=&\frac{1}{p}\|(u_n-u,v_n-v)\|_p^p+\frac{1}{q}\|(u_n-u,v_n-v)\|_q^q\nonumber\\
& &-\frac{1}{r}\int_{\Omega}(\lambda f(x)(u_n-u)^r+\mu g(x)(v_n-v)^r)dx
\end{eqnarray}
and 
\begin{eqnarray}\label{PS3}
0&=&\langle I'_{\alpha,\beta}(u_n,v_n),(u_n-u,v_n-v) \rangle\nonumber\\
&=&\langle I'_{\alpha,\beta}(u_n,v_n)-I'_{\alpha,\beta}(u,v),(u_n-u,v_n-v) \rangle\nonumber\\
&=&\|(u_n-u,v_n-v)\|_p^p+\|(u_n-u,v_n-v)\|_q^q-\int_{\Omega}(\lambda f(x)(u_n-u)^r+\mu g(x)(v_n-v)^r)dx+o(1).\nonumber\\
\end{eqnarray}
Now without loss of generality, we let
$$\|(u_n-u,v_n-v)\|_p^p=c'+o(1),~\|(u_n-u,v_n-v)\|_q^q=d'+o(1)$$
and therefore $$\int_{\Omega}(\lambda f(x)(u_n-u)^r+\mu g(x)(v_n-v)^r)dx=c'+d'+o(1).$$
Now if $c'=0$ the proof is immediate. On the contrary, we assume that $c'>0$.
\begin{eqnarray}
\left(\frac{c'}{2}\right)^{\frac{p}{p^*}}&\leq&\left(\frac{c'+d'}{2}\right)^{\frac{p}{p^*}}\nonumber\\
&=& \lim_{n\rightarrow\infty}\int_{\Omega}(\lambda f(x)(u_n-u)^r+\mu g(x)(v_n-v)^r)dx\nonumber\\
&\leq&M' \lim_{n\rightarrow\infty}\int_{\Omega}\left(\lambda|u_n-u|^{r}+\mu|v_n-v|^{r}\right)dx\nonumber\\
&\leq&M'\lim_{n\rightarrow\infty}|\Omega|^{\frac{1}{2-\alpha-\beta}-\frac{1}{r}}S_{\alpha,\beta}^{-\frac{r}{p}}\|(u_n-u,v_n-v)\|_p^{r}\nonumber\\
&=&M'|\Omega|^{\frac{1}{p}-\frac{1}{r}}S_{\alpha,\beta}^{-\frac{r}{p}} (\lambda^{\frac{r}{r-p}}+\mu^{\frac{r}{r-p}})c'^{\frac{r}{p}}.\nonumber
\end{eqnarray}
Thus, $$c'\geq\frac{S_{\alpha,\beta}^{\frac{r}{r-p}}}{\{2M'(\lambda^{\frac{r}{r-p}}+\mu^{\frac{r}{r-p}})\}^{\frac{p}{r-p}}|\Omega|^{\frac{1}{r}}}=\frac{S_{\alpha,\beta}^{\frac{r}{r-p}}}{\Lambda}.$$
Therefore from \eqref{PS2}, \eqref{PS3} and $(u,v)\in\mathcal{N}_{\alpha,\beta}\bigcup\{(0,0)\}$ we have
\begin{eqnarray}
c'&=&I_{\alpha,\beta}(u,v)+\frac{c'}{p}+\frac{d'}{q}-\frac{c'+d'}{r}\nonumber\\
&\geq&\frac{S_{\alpha,\beta}^{\frac{r}{r-p}}}{\Lambda}-\nu A_0\left[\left(\frac{1-\alpha}{2-\alpha-\beta}\right)^{\frac{p}{p+\alpha+\beta-2}}+\left(\frac{1-\beta}{2-\alpha-\beta}\right)^{\frac{p}{p+\alpha+\beta-2}}\right]\nonumber
\end{eqnarray}
which contradicts $c'<\frac{S_{\alpha,\beta}^{\frac{r}{r-p}}}{\Lambda}-\nu A_0\left[\left(\frac{1-\alpha}{2-\alpha-\beta}\right)^{\frac{p}{p+\alpha+\beta-2}}+\left(\frac{1-\beta}{2-\alpha-\beta}\right)^{\frac{p}{p+\alpha+\beta-2}}\right]$. Thus $c'=0$ and hence $(u_n,v_n)\rightarrow(u,v)$ in $X$.
\end{proof}
\noindent We will now see the proof of the existence of a local minimizer for $I_{\alpha,\beta}$ in $\mathcal{N}_{\alpha,\beta}^+$.
\begin{lemma}\label{locmin}
There exists $\Lambda^*>0$ such that $\nu\left[\left(\frac{1-\alpha}{2-\alpha-\beta}\right)^{\frac{p}{p+\alpha+\beta-2}}+\left(\frac{1-\beta}{2-\alpha-\beta}\right)^{\frac{p}{p+\alpha+\beta-2}}\right]\in (0,\Lambda^*)$, $I_{\alpha,\beta}$ has a minimizer $(u_{\nu},v_{\nu})\in\mathcal{N}_{\alpha,\beta}^+$ and it satisfies 
\begin{align*}
&(i)~I_{\alpha,\beta}(u_{\nu},v_{\nu})=i^{+}~\text{is a weak solution to the problem \eqref{main}}\\
&(ii)~I_{\alpha,\beta}(u_{\nu},v_{\nu})\rightarrow 0~\text{and}~\|(u_{\nu},v_{\nu})\|_p\rightarrow 0,~\|(u_{\nu},v_{\nu})\|_q\rightarrow 0~\text{as}~\nu\rightarrow 0.
\end{align*}
\end{lemma}
\begin{proof}
For the proof of $(i)$ we follow Hsu \cite{hsu1}, Theorem 4.2. Since $i^+=\inf_{(u,v)\in\mathcal{N}_{\alpha,\beta}}\{I_{\alpha,\beta}(u,v)\}$, there exists a sequence $(u_n,v_n)\in\mathcal{N}_{\alpha,\beta}$ such that $I_{\alpha,\beta}(u_n,v_n)\rightarrow i^+$ and $I'_{\alpha,\beta}(u_n,v_n)\rightarrow 0$ in $X^*$ as $n\rightarrow\infty$. Since the functional $I_{\alpha,\beta}$ is coercive and therefore $(u_n,v_n)$ is bounded in $X$. Thus there exists a subsequence of $(u_n,v_n)$, still denoted as $(u_n,v_n)$, such that $((u_n,v_n))\rightharpoonup (u,v)\in X$. So we have 
$$u_n\rightharpoonup u,~v_n\rightharpoonup v,$$
$$u_n\rightarrow u,~v_n\rightarrow v~\text{a.e.in}~\Omega,$$
$$u_n\rightarrow u,~v_n\rightarrow v~\text{in}~L^s(\Omega)~\text{for}~1\leq s <p^*$$
as $n\rightarrow\infty$. This implies 
$$\frac{2\nu}{2-\alpha-\beta}\int_{\Omega}h(x)u_n^{1-\alpha}v_n^{1-\beta}dx\rightarrow \frac{2\nu}{2-\alpha-\beta}\int_{\Omega}h(x)u^{1-\alpha}v^{1-\beta}dx.$$
Clearly $(u,v)$ is a weak solution of \eqref{main}. Also since $(u_n,v_n)\in\mathcal{N}_{\alpha,\beta}$ we have 
\begin{eqnarray}
L_{\alpha,\beta}^{\nu}(u_n,v_n)&=&\frac{r(2-\alpha-\beta)}{2\nu(r-2+\alpha+\beta)}\left(\frac{1}{p}-\frac{1}{r}\right)\|(u_n,v_n)\|_p^p\nonumber\\
&&+\frac{r(2-\alpha-\beta)}{2\nu(r-2+\alpha+\beta)}\left(\frac{1}{q}-\frac{1}{r}\right)\|(u_n,v_n)\|_q^q-\frac{r(2-\alpha-\beta)}{2\nu(r-2+\alpha+\beta)}I_{\alpha,\beta}(u_n,v_n)\nonumber
\end{eqnarray}
where $L_{\alpha,\beta}^{\nu}(u_n,v_n)=\int_{\Omega}h(x)u_n^{1-\alpha}v_n^{1-\beta}dx$. Also
\begin{eqnarray}
L_{\alpha,\beta}^{\nu}(u_n,v_n)&\geq&\frac{r(2-\alpha-\beta)}{2\nu(r-2+\alpha+\beta)}\left(\frac{1}{p}-\frac{1}{r}\right)\|(u,v)\|_p^p\nonumber\\
&&+\frac{r(2-\alpha-\beta)}{2\nu(r-2+\alpha+\beta)}\left(\frac{1}{q}-\frac{1}{r}\right)\|(u,v)\|_q^q-\frac{r(2-\alpha-\beta)}{2\nu(r-2+\alpha+\beta)}i^+\nonumber\\
&\geq&-\frac{r(2-\alpha-\beta)}{2\nu(r-2+\alpha+\beta)}i^+>0\nonumber
\end{eqnarray}
where we have used the lower-semicontinuity of $\|\cdot\|_p$, $\|\cdot\|_q$ and $i^+<0$. Therefore $(u,v)\neq (0,0)$. Thus we have a nontrivial weak solution. \\
{\it Claim}:~We now claim that $(u_n,v_n)\rightarrow(u,v)$ in $X$ and $I_{\alpha,\beta}(u,v)=i^+$.\\
For any $(u_0,v_0)\in\mathcal{N}_{\alpha,\beta}$ we have
\begin{eqnarray}
L_{\alpha,\beta}^{\nu}(u_0,v_0)&=&\frac{r(2-\alpha-\beta)}{2\nu(r-2+\alpha+\beta)}\left(\frac{1}{p}-\frac{1}{r}\right)\|(u_0,v_0)\|_p^p\nonumber\\
&&+\frac{r(2-\alpha-\beta)}{2\nu(r-2+\alpha+\beta)}\left(\frac{1}{q}-\frac{1}{r}\right)\|(u_0,v_0)\|_q^q-\frac{r(2-\alpha-\beta)}{2\nu(r-2+\alpha+\beta)}I_{\alpha,\beta}(u_0,v_0).\nonumber
\end{eqnarray}
Thus 
\begin{eqnarray}
i^+&\leq&I_{\alpha,\beta}(u,v)\nonumber\\
&\leq&\lim_{n\rightarrow\infty}\left[\left(\frac{1}{p}-\frac{1}{r}\right)\|(u_{n},v_{n})\|_p^p+\left(\frac{1}{q}-\frac{1}{r}\right)\|u_{n},v_{n}\|_q^q\right.\nonumber\\
& &\left.-\frac{2\nu}{2-\alpha-\beta}L_{\alpha,\beta}^{\nu}(u_n,v_n)\right]\nonumber\\
&=&I_{\alpha,\beta}(u,v)=i^+.\nonumber
\end{eqnarray}
Thus $I_{\alpha,\beta}(u,v)=i^+$. This also implies that $(u_n,v_n)\rightarrow(u,v)$ in $X$.\\
For the proof of $(ii)$ let $(u_{\nu},v_{\nu})\in\mathcal{N}_{\alpha,\beta}^{+}$. From Lemmas \ref{empty}, \ref{LBofI2}  we have that 
$$0>I_{\alpha,\beta}(u_{\nu},v_{\nu})\geq -\nu A_0\left[\left(\frac{1-\alpha}{2-\alpha-\beta}\right)^{\frac{p}{p+\alpha+\beta-2}}+\left(\frac{1-\beta}{2-\alpha-\beta}\right)^{\frac{p}{p+\alpha+\beta-2}}\right].$$
Therefore it is obvious that as $\nu\rightarrow 0$ we have $I_{\alpha,\beta}(u_{\nu},v_{\nu})\rightarrow 0$.\\
Further we have
\begin{eqnarray}
0&=&\lim_{\nu\rightarrow 0}I_{\alpha,\beta}(u_{\nu},v_{\nu})\nonumber\\
&=&\lim_{\nu\rightarrow 0}\left[\left(\frac{1}{p}-\frac{1}{r}\right)\|(u_{\nu},v_{\nu})\|_p^p+\left(\frac{1}{q}-\frac{1}{r}\right)\|u_{\nu},v_{\nu}\|_q^q\right.\nonumber\\
& &\left.-\frac{2\nu}{2-\alpha-\beta}\int_{\Omega}h(x)u_{\nu}^{1-\alpha}v_{\nu}^{1-\beta}dx\right].\nonumber
\end{eqnarray}
As seen earlier that the functional $I_{\alpha,\beta}$ is coercive over $\mathcal{N}_{\alpha,\beta}^+$ and therefore $(u_{\nu},v_{\nu})$ is bounded. Also using the fact $\lim_{\nu\rightarrow 0}\frac{2\nu}{2-\alpha-\beta}\int_{\Omega}h(x)u_{\nu}^{1-\alpha}v_{\nu}^{1-\beta}dx=0$ we clearly have
$$\lim_{\nu\rightarrow 0}\|(u_{\nu},v_{\nu})\|_p^p=0=\lim_{\nu\rightarrow 0}\|(u_{\nu},v_{\nu})\|_q^q.$$
\end{proof}
\begin{remark}\label{estirem}
For $\epsilon>0$ we define
$$u_{\epsilon}(x)=\frac{\eta(x)}{(\epsilon+|x|^{\frac{p}{p-1}})^{\frac{N-p}{p}}},~~v_{\epsilon}(x)=\frac{u_{\epsilon}(x)}{|u_{\epsilon}(x)|_{p^*}}$$
where $\eta(x)\in C_0^{\infty}(\Omega)$ is a radially symmetric function defined by
 \[\eta(x)=\begin{cases}
1& |x|<\rho_0 \\
0 & |x|>2\rho_0\\
0\leq\eta(x)\leq 1 &~\text{otherwise}.
\end{cases}\]
Further let $|\nabla\eta|\leq C$, where $\rho_0$ is such that $B(0,2\rho_0)\subset\Omega$.
 Then $\int_{\Omega}|u_{\epsilon}|^{p^*}dx=1$ and we have the following estimates
 \[\int_{\Omega}|u_{\epsilon}|^{t}dx=\begin{cases}
 C_1 \epsilon^{\frac{N(p-1)-t(N-p)}{p}}+O(1)& t>\frac{N(p-1)}{N-p} \\
 C_1|\ln\epsilon|+O(1) & t=\frac{N(p-1)}{N-p} \\
 O(1)\leq\eta(x)\leq 1 & t<\frac{N(p-1)}{N-p} .
 \end{cases}\]
 Therefore in particular we have 
 $$\int_{\Omega}|\nabla u_{\epsilon}|^pdx=K_2\epsilon^{\frac{p-N}{p}}+O(1)$$
 and
 $$\left(\int_{\Omega}|u_{\epsilon}|^{p^*}dx\right)^{\frac{1}{p^*}}=K_3\epsilon^{\frac{p-N}{p}}+O(1)$$
 where $K_1, K_2, K_3>0$ independent of $\epsilon$. Further there exists $\epsilon_0$ such that $S$, the best sobolev constant, is close to $\frac{K_2}{K_3}$ for every $0<\epsilon<\epsilon_0$. In other words we will take $S=\frac{K_2}{K_3}$.
\end{remark}
We now prove the following lemma which will be used in guaranteeing the multiplicity of solutions.
\begin{lemma}\label{auxlem1}
There exists $\epsilon_1$, $\Lambda^*$, $\sigma(\epsilon)>0$ such that for $\epsilon\in(0,\epsilon_1)$,\\ $\nu\left[\left(\frac{1-\alpha}{2-\alpha-\beta}\right)^{\frac{p}{p+\alpha+\beta-2}}+\left(\frac{1-\beta}{2-\alpha-\beta}\right)^{\frac{p}{p+\alpha+\beta-2}}\right]\in (0,\Lambda^*)$ and $\sigma\in(0,\sigma(\epsilon))$, we have 
$$\sup_{t\geq 0}I_{\alpha,\beta}(t_{\epsilon}\sqrt[p]{\nu}v_{\epsilon},t_{\epsilon}\sqrt[p]{\nu}v_{\epsilon})<c_{\alpha,\beta}-\sigma,$$
where $c_{\alpha,\beta}=\frac{r-p}{rp}S^{\frac{r}{r-p}}-\nu A_0\left[\left(\frac{1-\alpha}{2-\alpha-\beta}\right)^{\frac{p}{p+\alpha+\beta-2}}+\left(\frac{1-\beta}{2-\alpha-\beta}\right)^{\frac{p}{p+\alpha+\beta-2}}\right]$.
\end{lemma}
\begin{proof}
Define
\begin{eqnarray}
a(t)&=&I_{\alpha,\beta}(t\sqrt[p]{\nu}v_{\epsilon},t\sqrt[p]{\nu}v_{\epsilon})\nonumber\\
&=&\frac{t^p}{p}\nu\int_{\Omega}|\nabla v_{\epsilon}|^pdx+\frac{t^q}{q}(2\nu^{\frac{q}{p}})\int_{\Omega}|\nabla v_{\epsilon}|^qdx\nonumber\\
&&-\frac{1}{r}\int_{\Omega}(\lambda f(x)+\mu g(x))(tv_{\epsilon}\nu^{\frac{1}{p}})^rdx-\frac{2\nu^{\frac{p-\alpha-\beta+2}{p}} t^{2-\alpha-\beta}}{2-\alpha-\beta}\int_{\Omega}h(x)v_{\epsilon}^{2-\alpha-\beta}dx.\nonumber
\end{eqnarray}
Clearly $a(0)=0$, $\lim_{t\rightarrow\infty}a(t)=-\infty$. Therefore there exists $t_{\epsilon}>0$ such that 
$$I_{\alpha,\beta}(t_{\epsilon}\sqrt[p]{\nu}v_{\epsilon},t_{\epsilon}\sqrt[p]{\nu}v_{\epsilon})=\sup_{t\geq 0}I_{\alpha,\beta}(t\sqrt[p]{\nu}v_{\epsilon},t\sqrt[p]{\nu}v_{\epsilon}).$$
This yields that 
\begin{eqnarray}\label{equ1}
(2\nu)t_{\epsilon}^{p-1}\int_{\Omega}|\nabla v_{\epsilon}|^{p}dx+2t_{\epsilon}^{q-1}\nu^{\frac{q}{p}}\int_{\Omega}|\nabla v_{\epsilon}|^qdx&=&t_{\epsilon}^{r-1}\int_{\Omega}(\lambda f(x)+\mu g(x))(v_{\epsilon}\nu^{\frac{1}{p}})^rdx\nonumber\\
&&+2\nu^{\frac{p-\alpha-\beta+2}{p}} t_{\epsilon}^{1-\alpha-\beta}\int_{\Omega}h(x)v_{\epsilon}^{2-\alpha-\beta}dx.\nonumber\\
\end{eqnarray}
From \eqref{equ1} we have the following
\begin{eqnarray}\label{equ2}
t_{\epsilon}^{p+\alpha+\beta-2}\int_{\Omega}|\nabla v_{\epsilon}|^pdx&\leq&t_{\epsilon}^{r+\alpha+\beta-2}\int_{\Omega}(\lambda f(x)+\mu g(x))(v_{\epsilon}\nu^{\frac{1}{p}})^rdx\nonumber\\
&&+2\nu^{\frac{p-\alpha-\beta+2}{p}} \int_{\Omega}h(x)v_{\epsilon}^{2-\alpha-\beta}dx.\nonumber\\
\end{eqnarray} 
and 
\begin{eqnarray}\label{equ3}
(2\nu)t_{\epsilon}^{p-q}\int_{\Omega}|\nabla v_{\epsilon}|^{p}dx+2\nu^{\frac{q}{p}}\int_{\Omega}|\nabla v_{\epsilon}|^qdx&\geq&t_{\epsilon}^{r-q}\int_{\Omega}(\lambda f(x)+\mu g(x))(v_{\epsilon}\nu^{\frac{1}{p}})^rdx.\nonumber\\
\end{eqnarray}
From the estimates for $u_{\epsilon}$ obtained in the Remark \ref{estirem}, i.e.
$$\int_{\Omega}|\nabla v_{\epsilon}|^pdx=S+O(\epsilon^{\frac{N-p}{p}}),~\int_{\Omega}| v_{\epsilon}|^rdx=O(\epsilon^{\frac{r(N-p)}{p^2}}),~\int_{\Omega}| v_{\epsilon}|^{2-\alpha-\beta}dx=O(\epsilon^{\frac{(2-\alpha-\beta)(N-p)}{p^2}}).$$ 
From \eqref{equ1} it very easily follows now that 
\begin{eqnarray}
t_{\epsilon}^{p+\alpha+\beta-2}(S+O(\epsilon^{\frac{N-p}{p}}))&\leq&CM't_{\epsilon}^{r+\alpha+\beta-2}+2M\nu^{\frac{p-\alpha-\beta+2}{p}}O(\epsilon^{\frac{(2-\alpha-\beta)(N-p)}{p^2}})\nonumber\\
\end{eqnarray} 
where we have use the estimate 
$$\int_{\Omega}(\lambda f(x)+\mu g(x))v_{\epsilon}^rdx\leq CM'\|v_{\epsilon}\|_{p^*}^r=CM'.$$
Thus, there exists $T_1>0$, $\epsilon_1>0$ such that for any $\epsilon\in(0,\epsilon_1)$, we have $t_{
\epsilon}\geq T_1$. 
Likewise we have
\begin{eqnarray}\label{eq2}
(2\nu)t_{\epsilon}^{p-q}(S+O(\epsilon^{\frac{N-p}{p}}))+2C\nu^{\frac{q}{p}}&\geq&Ct_{\epsilon}^{2-\alpha-\beta-q}.\nonumber\\
\end{eqnarray}
Then, there exists $T_2>0$, $\epsilon_2>0$ such that for any $\epsilon\in(0,\epsilon_2)$, we have $t_{\epsilon}\leq T_2$. Let $\tilde{\epsilon}=\min\{\epsilon_1,\epsilon_2\}$. Then for any $\epsilon\in(0,\tilde{\epsilon})$ we have $T_1\leq t_{\epsilon}\leq T_2$. 
Consider
$$b(t)=\frac{t^p}{p}\nu\int_{\Omega}|\nabla v_{\epsilon}|^{p}dx-\frac{1}{r}\int_{\Omega}(\lambda f(x)+\mu g(x))(tv_{\epsilon}\nu^{\frac{1}{p}})^rdx.$$
Then a simple calcultaion gives  
\begin{eqnarray}
\sup_{t\geq 0}b(t)&=&\frac{r-p}{rp}S^{\frac{r}{r-p}}+O(\epsilon^{\frac{N-p}{p}}).\nonumber
\end{eqnarray}
Therefore, for any $\epsilon\in(0,\tilde{\epsilon})$, we have 
\begin{eqnarray}
a(t_{\epsilon})&=&b(t_{\epsilon})+\frac{t_{\epsilon}^q}{q}(\nu^{\frac{q}{p}})\int_{\Omega}|\nabla v_{\epsilon}|^qdx\nonumber\\
&&-\frac{\nu^{\frac{p-\alpha-\beta+2}{p}} t_{\epsilon}^{2-\alpha-\beta}}{2-\alpha-\beta}\int_{\Omega}h(x)v_{\epsilon}^{2-\alpha-\beta}dx\nonumber\\
&\leq&\frac{r-p}{rp}S^{\frac{r}{r-p}}+O(\epsilon^{\frac{N-p}{p}})+\frac{t_{\epsilon}^q}{q}(2\nu^{\frac{q}{p}})\int_{\Omega}|\nabla v_{\epsilon}|^qdx\nonumber\\
&&-\frac{\nu^{\frac{p-\alpha-\beta+2}{p}} t_{\epsilon}^{2-\alpha-\beta}}{2-\alpha-\beta}\int_{\Omega}h(x)v_{\epsilon}^{2-\alpha-\beta}dx\nonumber\\
&\leq&\frac{r-p}{rp}S^{\frac{r}{r-p}}+O(\epsilon^{\frac{N-p}{p}})+\frac{T_2^q}{q}(2\nu^{\frac{q}{p}})\int_{\Omega}|\nabla v_{\epsilon}|^qdx\nonumber\\
&&-\frac{\nu^{\frac{p-\alpha-\beta+2}{p}} T_1^{2-\alpha-\beta}}{2-\alpha-\beta}\int_{\Omega}h(x)v_{\epsilon}^{2-\alpha-\beta}dx\nonumber\\
&\leq&\frac{r-p}{rp}S^{\frac{r}{r-p}}+O(\epsilon^{\frac{N-p}{p}})+O(\epsilon^{\frac{q(N-p)}{p^2}})-O(\epsilon^{\frac{(2-\alpha-\beta)(N-p)}{p^2}}).\nonumber
\end{eqnarray}
From the assumptions in the problem in \eqref{main} we also have 
$$0<\frac{(2-\alpha-\beta)(N-p)}{p^2}<\frac{q(N-p)}{p^2}<\frac{N-p}{p}.$$
Therefore, one can choose $\epsilon_1>0$, sufficiently small, $\Lambda^*$, $\sigma(\epsilon)>0$ such that for $\epsilon\in(0,\epsilon_1)$, $\nu\left[\left(\frac{1-\alpha}{2-\alpha-\beta}\right)^{\frac{p}{p+\alpha+\beta-2}}+\left(\frac{1-\beta}{2-\alpha-\beta}\right)^{\frac{p}{p+\alpha+\beta-2}}\right]\in (0,\Lambda^*)$ and $\sigma\in(0,\sigma(\epsilon))$, we have 
$$O(\epsilon^{\frac{N-p}{p}})+O(\epsilon^{\frac{q(N-p)}{p^2}})-O(\epsilon^{\frac{(2-\alpha-\beta)(N-p)}{p^2}})<-A_0\nu\left[\left(\frac{1-\alpha}{2-\alpha-\beta}\right)^{\frac{p}{p+\alpha+\beta-2}}+\left(\frac{1-\beta}{2-\alpha-\beta}\right)^{\frac{p}{p+\alpha+\beta-2}}\right]-\sigma.$$
\end{proof}
\section{Few useful lemmas}
This section is devoted to recall and prove some important lemmas which are crucial to the proof of the main theorem. We first consider a submanifold of $\mathcal{N}_{\alpha,\beta}^-$ defined as follows.
$$\mathcal{N}_{\alpha,\beta}^-(c_{\alpha,\beta})=\{(u,v)\in\mathcal{N}_{\alpha,\beta}^-:I_{\alpha,\beta}(u,v)\leq c_{\alpha,\beta}\}.$$
The main result which we will prove in this section is that the problem in \eqref{main} admits at least $\text{cat}(\Omega)$ number of solutions in this set.
\begin{definition}\label{cat}
(a)~For a topological space $X$, we say that a non-empty, closed subspace $Y\subset X$ is contractible to a point if and only if there exists a continuous mapping $$\xi:[0,1]\times Y\rightarrow X$$
such that for some $x_0\in X$. there hold
$$\xi(0,x)=x,~\text{for all}~x\in Y$$
and 
$$\xi(1,x)=x_0,~\text{for all}~x\in Y.$$
(b)~If $Y$ is closed subset of a topological space $X$, $\text{cat}_{X}(Y)$ denotes Lusternik-Schnirelman category of $Y$, i.e., the least number of closed and contractible sets in $X$ which cover $Y$.
\end{definition}
\noindent We now state an auxilliary lemma which can be found in the form of Theorem 1 in \cite{alv1}.
\begin{lemma}\label{mainauxthm}
Suppose that $X$ is a $C^{1,1}$ complete Riemanian manifold and $I\in C^1(X,\mathbb{R})$. Assume that for $c_0\in\mathbb{R}$ and $k\in\mathbb{N}$:
\begin{eqnarray}
(i)~I~\text{satisfies the}~(PS)_c~\text{condition for}~c\leq c_0\nonumber\\
(ii)~\text{cat}{(u\in X:I(u)\leq c_0)}\geq k.
\end{eqnarray}
Then $I$ has at least $k$ critical points in ${u\in X:I(u)\leq c_0}$.
\end{lemma}
\noindent The following lemma is a standard one and can be proved if one works in the lines of the argument in \cite{will}.
\begin{lemma}\label{catlemma}
Let $\{(u_n,v_n)\}\subset X$ be a nonnegative sequence  of functions with $\int_{\Omega}(u_n^{r}+v_n^{r})dx=1$ and $\|(u_n,v_n)\|_{p}^p\rightarrow S_{\alpha,\beta}$. Then there exists a sequence $\{(y_n,\theta_n)\}\subset \mathbb{R}^N\times\mathbb{R}^+$ such that 
$$\omega_n(x)=(\omega_n^1(x),\omega_n^2(x))=\theta_n^{\frac{N}{r}}(u_n(\theta_n x+y_n),v_n(\theta_n x+y_n))$$
contains a convergent subsequence denoted again by $\{\omega_n\}$ such that 
$$\omega_n\rightarrow \omega~\text{in}~W^{1,p}(\mathbb{R}^N\times W^{1,p}(\mathbb{R}^N),$$
where $\omega=(\omega^1,\omega^2)>0$ in $\mathbb{R}^N$. Moreover, we have $\theta_n\rightarrow 0$ and $y_n\rightarrow y\in\overline{\Omega}$ as $n\rightarrow\infty$.
\end{lemma}
\noindent Upto translations, we assume that $0\in\Omega$. Moreover, we schoose $\delta>0$ small enough such that $B_{\delta}=\{x\in\mathbb{R}^N:\text{dist}(x,\partial\Omega)<\delta\}$ and the sets 
$$\Omega_{\delta}^+=\{x\in\mathbb{R}^N:\text{dist}(x,\partial\Omega)<\delta\},~\Omega_{\delta}^-=\{x\in\mathbb{R}^N:\text{dist}(x,\partial\Omega)>\delta\}$$
are both homotopically equivalent to $\Omega$. By using the idea of \cite{fan1} or \cite{li5} we define a continuous mapping $\tau:\mathcal{N}_{\alpha,\beta}^-\rightarrow\mathbb{R}^N$ by setting $$\tau(u,v)=\frac{\int_{\Omega}x(\lambda u^r+\mu v^r)dx}{\int_{\Omega}(\lambda u^r+\mu v^r)dx}.$$
\begin{remark}
As told before that the functional $I_{\alpha,\beta}$ is not a $C^1$-functional, we might fail to use some very useful techniques in variational techniques. For this we will define a {\it cut-off} functional using a subsolution (refer \cite{Evans} for a definition) to the system in \eqref{main}.
	Define,
\[   
\overline{f}(x,t,s) = 
\begin{cases}
f(x,t,s) &~\text{if}~t>\underline{u}, s>\underline{v}\\
f(x,t,\underline{v}) &~\text{if}~t>\underline{u}, s\leq\underline{v}\\
f(x,\underline{u},s) &~\text{if}~t\leq\underline{u}, s>\underline{v}\\
f(x,\underline{u}_{\lambda},s)&~\text{if}~t\leq \underline{u}, s\leq \underline{v}
\end{cases}\]
where $f(x,t,s)=\lambda f(x)t^{r-1}+\mu g(x)s^{r-1}+\nu\frac{1-\alpha}{2-\alpha-\beta}h(x)t^{-\alpha}s^{1-\beta}+\nu\frac{1-\beta}{2-\alpha-\beta}h(x)t^{1-\alpha}s^{-\beta}$ is a subsolution to \eqref{main} (the existence of such a solution can be guaranteed by the previous sections by taking $\lambda=\mu=0$ in \eqref{main}). Let $\overline{F}(x,t,s)=\int_{0}^t\int_0^s\overline{f}(x,t,s)dsdt$ and $(\underline{u},\underline{v})$. 
Define a function $\overline{I}_{\lambda}:X \times X\rightarrow\mathbb{R}$ as follows.
\begin{align}\label{fang_fnal}
\overline{I}_{\lambda}(u,v)&=\frac{1}{p}\|(u,v)\|_p^p+\frac{1}{q}\|(u,v)\|_q^q-\int_{\Omega}\overline{F}(x,u,v)dx.
\end{align}
The functional is $C^1$ (the proof follows the arguments of the Lemma 6.4 in the Appendix of \cite{ghosh}) and weakly lower semicontinuous. The way the functional has been defined, it is not difficult to see that the critical points of the fnctional corresponding to the problem \eqref{main} and that of the cut-off functional are the same.
\end{remark}
\begin{remark}
We will continue to name the cut-off functional $I_{\lambda}$ as $I_{\alpha,\beta}$.
\end{remark}
\noindent We then have the following result.
\begin{lemma}\label{taulemma}
There exists $\Lambda^*$ such that if $\nu\left[\left(\frac{1-\alpha}{2-\alpha-\beta}\right)^{\frac{p}{p+\alpha+\beta-2}}+\left(\frac{1-\beta}{2-\alpha-\beta}\right)^{\frac{p}{p+\alpha+\beta-2}}\right]\in(0,\Lambda^*)$ and $(u,v)\in\mathcal{N}_{\alpha,\beta}^-(c_{\alpha,\beta})$, then $\tau(u,v)\in\Omega_{\delta}^+$.
\end{lemma}
\begin{proof}
Let us assume that there exists sequences $\nu_n\rightarrow 0$ and $\{(u_n,v_n)\}$ such that $\tau(u_n,v_n)\not\in\Omega_{\delta}^+$. By using the tactics in one of the previous lemmas (\ref{PSbounded}) we conclude the boundedness of the sequence  $\{(u_n,v_n)\}$ in $X$. Then we have 
$$\nu_n\int_{\Omega}h(x)u_n^{1-\alpha}v_n^{1-\beta}dx\rightarrow 0~\text{as}~n\rightarrow\infty.$$
Therefore we get
\begin{eqnarray}
I_{\alpha,\beta}(u_n,v_n)&=&\left(\frac{1}{p}-\frac{1}{r}\right)\|(u_n,v_n)\|_p^p+\left(\frac{1}{q}-\frac{1}{r}\right)\|(u_n,v_n)\|_{q}^q+o(1)\leq c_{\alpha,\beta}^{\nu_n}+o(1)\nonumber
\end{eqnarray}
and 
\begin{eqnarray}
\left(\frac{1}{p}-\frac{1}{r}\right)\|(u_n,v_n)\|_p^p&\leq& c_{\alpha,\beta}^{\nu_n}+o(1)\leq\frac{S^{\frac{r}{r-p}}}{\Lambda}+o(1).\nonumber
\end{eqnarray}
This implies that 
\begin{eqnarray}\label{esti1}\|(u_n,v_n)\|_p^p&\leq&\frac{rp}{r-p}\frac{S^{\frac{r}{r-p}}}{\Lambda}+o(1).\end{eqnarray}
Since $\{(u_n,v_n)\}\subset\mathcal{N}_{\alpha,\beta}^-(c_{\alpha,\beta}^{\nu_n})\subset\mathcal{N}_{\alpha,\beta}^-$, we have
\begin{eqnarray}\label{esti2}\|(u_n,v_n)\|_p^p&\leq&\int_{\Omega}(\lambda f(x)u_n^r+\mu g(x)v_n^r)dx+o(1)\leq M'|(u_n,v_n)|_{p^*}^r+o(1).\nonumber\\\end{eqnarray}
By \eqref{esti1} and \eqref{esti2} we get 
\begin{eqnarray}
S_{\alpha,\beta}&\leq&\frac{\|(u_n,v_n)\|_p^p}{\{\int_{\Omega}(u_n^{p^*}+v_n^{p^*})dx\}^{\frac{p}{p^*}}}\nonumber\\
&\leq&C\|(u_n,v_n)\|_p^p\nonumber\\
&\leq&S_{\alpha,\beta}+o(1)
\end{eqnarray}
which implies that $\|(u_n,v_n)\|_p^p\rightarrow CS_{\alpha,\beta}^{\frac{p}{r-p}}$ and $\int_{\Omega}(\lambda f(x)u_n^r+\mu g(x)v_n^r)dx\rightarrow C'S_{\alpha,\beta}^{\frac{p}{r-p}}$.\\
Define 
$$(\xi_n,\eta_n)=\left(\frac{u_n}{(\int_{\Omega}(\lambda u_n^r+\mu v_n^r)dx)^{1/r}},\frac{v_n}{(\int_{\Omega}(\lambda u_n^r+\mu v_n^r)dx)^{1/r}}\right).$$
Clearly,
$$\int_{\Omega}(\lambda\xi_n^r+\mu\eta_n^r)dx=1$$ and $$\int_{\Omega}(|\nabla \xi_n|^p+|\eta_n|^pdx)\rightarrow S_{\alpha,\beta}^{\frac{p}{r-p}\frac{r-1}{r}},~\text{as}~n\rightarrow\infty.$$
\noindent From the Lemma \ref{catlemma}, there exists a sequence $\{(y_n,\theta_n)\}\subset\mathbb{N}\times\mathbb{R}^+$ such that $\theta_n\rightarrow 0$, $y_n\rightarrow y\in\overline{\Omega}$ and 
$\omega(x)=(\omega_n^1(x),\omega_n^2(x))=\theta_n^{\frac{N}{r}}({\xi}_n(\theta_n x+y_n),{\eta}_n(\theta_n x+y_n))\rightarrow(\omega_1,\omega_2)$ with $\omega_1, \omega_2>0$ in $\mathbb{R}^N$ as $n\rightarrow\infty$.\\
Let $\chi\in C_0^{\infty}(\mathbb{R}^N)$ such that $\chi(x)=x$ in $\Omega$. Then we guarantee that
\begin{eqnarray}
\tau(u_n,v_n)&=&\frac{\int_{\Omega}\chi(x)(\lambda u_n^r+\mu v_n^r)dx}{\int_{\Omega}(\lambda u_n^r+\mu v_n^r)dx}\nonumber\\
&=&\int_{\Omega}\theta_n^{N}\chi(\theta_n x+y_n)(\lambda\xi_n^r+\mu\eta_n^r)dx\nonumber\\
&=&\int_{\Omega}\chi(\theta_n x_n+y_n)(\lambda(\omega_n(x)^1)^r+\mu(\omega_n(x)^2)^r)dx.
\end{eqnarray}
By the lebesgue dominated convergence theorem we have 
$$\int_{\Omega}\chi(\theta_n x_n+y_n)(\lambda(\omega_n^1)^r+\mu(\omega_n^2)^r)dx\rightarrow y\in\overline{\Omega}$$
as $n\rightarrow\infty$. this implies that $\tau(x_n,y_n)\rightarrow y\in\overline{\Omega}$ as $n\rightarrow\infty$, which leads to a contradiction to our assumption. 
\end{proof}
The analysis done till now tells us that $\inf_{M_{\delta}}u_{\alpha,\beta}>0$ and $\inf_{M_{\delta}}v_{\alpha,\beta}>0$, thanks to the Lemma \ref{locmin} and the definition of $\Omega_{\delta}^-$. Note that $M_{\delta}=\{x\in\Omega:\text{dist}(x,\Omega_{\delta}^-)\leq\frac{\delta}{2}\}$ which is a compact set. Thus by the Lemma \ref{auxlem1} and using the idea of Lemma 3.4 of \cite{fan1}, Lemma 3.3 of \cite{chn1}, we can obtain a $\tilde{t}^->0$ such that 
$$(\tilde{t}^-\sqrt[p]{\nu}v_{\epsilon}(x-y),\tilde{t}\sqrt[p]{\nu}v_{\epsilon}(x-y))\in\mathcal{N}_{\alpha,\beta}(c_{\alpha,\beta}-\sigma)$$
uniformly in $y\in\Omega_{\delta}^-$. Further, by the lemma \ref{taulemma}, $\tau(\tilde{t}^-\sqrt[p]{\nu}v_{\epsilon}(x-y),\tilde{t}\sqrt[p]{\nu}v_{\epsilon}(x-y))\in\Omega_{\delta}^-$. Thus we can define a map $\gamma:\Omega_{\delta}^-\rightarrow\mathcal{N}_{\alpha,\beta}(c_{\alpha,\beta}-\sigma)^-$ by 
\[\gamma(y)=\begin{cases}
 (\tilde{t}^-\sqrt[p]{\nu}v_{\epsilon}(x-y),\tilde{t}\sqrt[p]{\nu}v_{\epsilon}(x-y)),& ~\text{if}~x\in B_{\delta}(y) \\
0,&~\text{otherwise}.
\end{cases}\]
We will denote by $\tau_{\alpha,\beta}$ the restriction of $\tau$ over $\mathcal{N}_{\alpha,\beta}^-(c_{\alpha,\beta}-\sigma)$.  Observe that $v_{\epsilon}$ is a radial function, therefore for each $y\in\Omega_{\delta}^-$, we have 
\begin{eqnarray}
(\tau_{\alpha,\beta}\circ\gamma)(y)&=&\frac{\int_{\Omega}x(\lambda (\tilde{t}^-\sqrt[p]{\nu}v_{\epsilon}(x-y))^r+\mu (\tilde{t}^-\sqrt[p]{\nu}v_{\epsilon}(x-y))^r)dx}{\int_{\Omega}(\lambda (\tilde{t}^-\sqrt[p]{\nu}v_{\epsilon}(x-y))^r+\mu (\tilde{t}^-\sqrt[p]{\nu}v_{\epsilon}(x-y))^r)dx}\nonumber\\
&=&\frac{\int_{\Omega}(y+z)(\tilde{t}^-)^r\nu^{\frac{r}{p}}(\lambda+\mu)v_{\epsilon}^rdz}{\int_{\Omega}(\tilde{t}^-)^r\nu^{\frac{r}{p}}(\lambda+\mu)v_{\epsilon}^rdz}\nonumber\\
&=&y.\nonumber
\end{eqnarray}
From \cite{fan1}, we define the map $H_{\alpha,\beta}:[0,1]\times\mathcal{N}_{\alpha,\beta}^-(c_{\alpha,\beta}-\sigma)\rightarrow\mathbb{R}^N$ by 
$$H_{\alpha,\beta}(t,z)=t\tau_{\alpha,\beta}(z)+(1-t)\tau_{\alpha,\beta}(z).$$
We then have the following lemma.
\begin{lemma}
To each $\epsilon\in(0,\epsilon_0)$, there exists $\Lambda^*>0$ such that if\\ $\nu\left[\left(\frac{1-\alpha}{2-\alpha-\beta}\right)^{\frac{p}{p+\alpha+\beta-2}}+\left(\frac{1-\beta}{2-\alpha-\beta}\right)^{\frac{p}{p+\alpha+\beta-2}}\right]\in (0,\Lambda^*)$, we have $H_{\alpha,\beta}([0,1]\times\mathcal{N}_{\alpha,\beta}^-(c_{\alpha,\beta}-\sigma))\subset\Omega_{\delta}^-$. 
\end{lemma}
\begin{proof}
We prove by contradiction. Let there exists sequences $t_n\in[0,1]$, $\nu_n\rightarrow 0$ and $z_n=(u_n,v_n)\in\mathcal{N}_{\alpha,\beta}^-(c_{\alpha,\beta}-\sigma)$ such that $H_{\alpha,\beta}(t_n,z_n)\not\in\Omega_{\delta}^+$ for all $n$. We can assume that $t_n\rightarrow t\in[0,1]$. Thus by Lemma \ref{locmin} $(ii)$ and similar argument in the proof of \ref{taulemma}, we have
$$H_{\alpha,\beta}(t_n,z_n)\rightarrow y\in\overline{\Omega}~\text{as}~n\rightarrow\infty$$
which leads to a contradiction.
\end{proof}
\noindent We now prove the main result of this article which roughly states that under certain assumptions on $\nu$ the problem in \eqref{main} admits at least $\text{cat}(\Omega)+1$ number of solutions.
\begin{lemma}\label{critpoint}
If $(u,v)$ is a critical point of $I_{\alpha,\beta}$ on $\mathcal{N}_{\alpha,\beta}^-$, then it is also a critical point of $I_{\alpha,\beta}$ in $X$.
\end{lemma}
\begin{proof}
We follow the proof of Lemma 4.1 in \cite{fan1} or of Lemma 4.1 in \cite{yin3}. Let $(u,v)$ be a crtical point of $I_{\alpha,\beta}$ in $\mathcal{N}_{\alpha,\beta}^-$. Then 
$$\langle I_{\alpha,\beta}'(u,v),(u,v)\rangle=0.$$
Define
\begin{eqnarray}
\psi_{\alpha,\beta}(u,v)&=&\langle I'_{\alpha,\beta}(u,v),(u,v)\rangle\nonumber\\
&=&\|(u,v)\|_p^p+\|(u,v)\|_q^q-\int_{\Omega}(\lambda f(x)u^r+\mu g(x)v^r)dx\nonumber\\
& &-\nu\int_{\Omega}h(x)u^{1-\alpha}v^{1-\beta}dx.\nonumber
\end{eqnarray}
Since we are now looking for minimizing $I$ over the entire space $X$, to which the Lagrange multiplier method comes to our rescue in finding a $\theta(\neq 0)\in\mathbb{R}$ such that 
\begin{eqnarray}\label{ref1}I_{\alpha,\beta}'(u,v)=\theta \psi'(u,v)\end{eqnarray}
where
\begin{eqnarray}
\psi_{\alpha,\beta}(u,v)&=&\langle I_{\alpha,\beta}'(u,v),(u,v)\rangle.\nonumber
\end{eqnarray}
Since, $(u,v)\in\mathcal{N}_{\alpha,\beta}^-$, we have from a simple computation that $\psi_{\alpha,\beta}(u,v)<0$. Consequently from \eqref{ref1} we have $ I_{\alpha,\beta}'(u,v)=0$.
\end{proof}
\begin{lemma}\label{convseq}
There exists $\Lambda^*>0$ such that any sequence $\{(u_n,v_n)\}\subset\mathcal{N}_{\alpha,\beta}^-$ with $I_{\mathcal{N}_{\alpha,\beta}^-}(u_n,v_n)\rightarrow c\in(-\infty,c_{\alpha,\beta})$ and $I'_{\mathcal{N}_{\alpha,\beta}^-}(u_n,v_n)\rightarrow 0$ contains a convergent subsequence for all $0<\nu\left[\left(\frac{1-\alpha}{2-\alpha-\beta}\right)^{\frac{p}{p+\alpha+\beta-2}}+\left(\frac{1-\beta}{2-\alpha-\beta}\right)^{\frac{p}{p+\alpha+\beta-2}}\right]<\Lambda^*$. 
\end{lemma}
\begin{proof}
From the Lagrange's multiplier method, there exists a sequence $(a_n)\subset\mathbb{R}$ such that 
$$\|I'_{\alpha,\beta}(u_n,v_n)-a_n\psi'_{\alpha,\beta}(u_n,v_n)\|_{X'}\rightarrow 0$$
as $n\rightarrow\infty$. Here
\begin{eqnarray}
\psi_{\alpha,\beta}(u_n,v_n)&=&\langle I'_{\alpha,\beta}(u_n,v_n),(u_n,v_n)\rangle\nonumber\\
&=&\|(u_n,v_n)\|_p^p+\|(u_n,v_n)\|_q^q-\int_{\Omega}(\lambda f(x)u_n^r+\mu g(x)v_n^r)dx\nonumber\\
& &-\nu\int_{\Omega}h(x)u_n^{1-\alpha}v_n^{1-\beta}dx.\nonumber
\end{eqnarray}
Then 
$$I_{\alpha,\beta}'(u_n,v_n)=a_n \psi_{\alpha,\beta}'(u_n,v_n)+o(1).$$
Since $(u_n,v_n)\in\mathcal{N}_{\alpha,\beta}^-\subset\mathcal{N}_{\alpha,\beta}$, by a simple computation we have 
$$\langle\psi_{\alpha,\beta}(u_n,v_n),(u_n,v_n)\rangle<0.$$
Now suppose $\langle\psi'_{\alpha,\beta}(u_n,v_n),(u_n,v_n)\rangle\rightarrow 0$, then we have
\begin{eqnarray}
(r-p)\|(u_n,v_n)\|_p^p+(r-q)\|(u_n,v_n)\|_q^q&=&\nu(1+\alpha+\beta)\int_{\Omega}h(x)u_n^{1-\alpha}v_n^{1-\beta}dx+o(1)\nonumber\\
&\leq&\nu(1+\alpha+\beta)M\left[\left(\frac{1-\alpha}{2-\alpha-\beta}\right)^{\frac{p}{p+\alpha+\beta-2}}\right.\nonumber\\
& &\left.+\left(\frac{1-\beta}{2-\alpha-\beta}\right)^{\frac{p}{p+\alpha+\beta-2}}\right]^{\frac{p+\alpha+\beta-2}{p}}\|(u_n,v_n)\|_p^{2-\alpha-\beta}\nonumber\\
& &+o(1)\nonumber
\end{eqnarray}
and
\begin{eqnarray}
(p+\alpha+\beta-2)\|(u_n,v_n)\|_p^p+(q+\alpha+\beta-2)\|(u_n,v_n)\|_q^q\nonumber\\
=(r+\alpha+\beta-2)\int_{\Omega}(\lambda f(x)u_n^r+\beta g(x)v_n^r)dx+o(1)\leq M'\|(u_n,v_n)\|_p^{p^*}+o(1)\nonumber
\end{eqnarray}
where we have used the H\"{o}lder inequlaity and the Sobolev embedding.
Then we have 
$$\|(u_n,v_n)\|_p\leq (\nu C_1)^{\frac{1}{p}}\left[\left(\frac{1-\alpha}{2-\alpha-\beta}\right)^{\frac{p}{p+\alpha+\beta-2}}+\left(\frac{1-\beta}{2-\alpha-\beta}\right)^{\frac{p}{p+\alpha+\beta-2}}\right]^{\frac{1}{p}}+o(1)$$
and 
$$\|(u_n,v_n)\|_p\geq C_2^{\frac{1}{p^*-p}}+o(1).$$
Now if we choose $\Lambda^*$ small enough, this cannot hold. Therefore let us assume that $\langle\psi_{\alpha,\beta}(u_n,v_n),(u_n,v_n)\rangle\rightarrow l < 0$, as $n\rightarrow\infty$. since $\langle I_{\alpha,\beta}(u_n,v_n),(u_n,v_n)\rangle=0$, we conlcude that $a_n\rightarrow 0$ and therefore $I_{\alpha,\beta}'(u_n,v_n)\rightarrow 0$. This gives us that 
$$I_{\alpha,\beta}(u_n,v_n)=c<c_{\alpha,\beta}~\text{and}~I_{\alpha,\beta}'(u_n,v_n)\rightarrow 0.$$
Therefore by the Lemma \ref{PSbounded} the proof is complete.
\end{proof}
\begin{lemma}\label{lust-schli}
Suppose that (C) holds and $\nu\left[\left(\frac{1-\alpha}{2-\alpha-\beta}\right)^{\frac{p}{p+\alpha+\beta-2}}+\left(\frac{1-\beta}{2-\alpha-\beta}\right)^{\frac{p}{p+\alpha+\beta-2}}\right]\in(0,\Lambda^*)$. Then $\text{cat}(\mathcal{N}_{\lambda,\mu}^-(c_{\lambda,\mu}-\sigma))\geq\text{cat}(\Omega)$.
\end{lemma}
\begin{proof}
Let $\text{cat}(\mathcal{N}_{\alpha,\beta}^-(c_{\alpha,\beta}-\sigma))=n$. Then, by the definition \ref{cat} of the category of a set in the sense of Lusternik-Schnirelman, we suppose that 
$$\mathcal{N}_{\alpha,\beta}^-(c_{\alpha,\beta}-\sigma)=A_1\cup A_2\cup...\cup A_n$$
where $A_j$, $j=1,2,...,n$ are closed and contractible in $\mathcal{N}_{\alpha,\beta}^-(c_{\alpha,\beta}-\sigma)$, i.e., there exists $h_j\in C([0,1]\times A_j, \mathcal{N}_{\alpha,\beta}^-(c_{\alpha,\beta}-\sigma))$ such that
$$h_j(0,z)=z,~h_j(1,z)=\Theta,~\text{for all}~z\in A_j,$$
where $\Theta\in A_j$ is fixed. Consider $B_j=\gamma^{-1}(A_j)$, $j=1,2,...,n$. Then the sets $B_j$ are closed 
$$\Omega_{\delta}^-=B_1\cup B_2\cup...\cup B_n.$$
We now define the deformation $g_j:[0,1]\times B_j\rightarrow\Omega_{\delta}^+$ by setting 
$$g_j(t,y)=H_{\alpha,\beta}(t,h_j(t,\gamma(y))).$$
for $\nu\left[\left(\frac{1-\alpha}{2-\alpha-\beta}\right)^{\frac{p+\alpha+\beta-2}{p}}+\left(\frac{1-\beta}{2-\alpha-\beta}\right)^{\frac{p+\alpha+\beta-2}{p}}\right]\in(0,\Lambda^*)$. Notice that 
$$g_j(0,y)=H_{\alpha,\beta}(0,h_j(0,\gamma(y)))=(\tau_{\alpha,\beta}\circ\gamma)(y)=y,~\text{for all}~y\in B_j$$
and
$$g_j(1,y)=H_{\alpha,\beta}(0,h_j(1,\gamma(y)))=\tau_{\alpha,\beta}(\Theta)\in\Omega_{\delta}^+,~\text{for all}~y\in B_j.$$
Thus the sets $B_j$, $j=1,2,...,n$ are contractible in $\Omega_{\delta}^{+}$. Therefore $\text{cat}(\mathcal{N}_{\alpha,\beta}^--\sigma)\geq\text{cat}_{\Omega_{\delta}^+}(\Omega_{\delta}^-)=\text{cat}(\Omega)$.
\end{proof}
\noindent {\bf Proof of Theorem \ref{mainresult}.}~By Lemmas \ref{PSbounded} and \ref{convseq}, the functional $I_{\alpha,\beta}$ satisfies the $(PS)_c$ condition for $c\in(-\infty,c_{\alpha,\beta})$. Then, by Lemma \ref{mainauxthm} and \ref{lust-schli}, we have $I_{\alpha,\beta}$ has at least $\text{cat}(\Omega)$ number of critical points in $\mathcal{N}_{\alpha,\beta}^-(c_{\alpha,\beta}-\sigma)$. By Lemma \ref{critpoint}, we have $I_{\alpha,\beta}$ has at least $\text{cat}(\Omega)$ number of critical points in $\mathcal{N}_{\alpha,\beta}^- $. Further, since $\mathcal{N}_{\alpha,\beta}^+\cap\mathcal{N}_{\alpha,\beta}^-=\phi$, the proof is now complete.
\section*{Acknowledgement}
The author thanks the SERB-MATRICS, India, for the financial assistanceship received to carry out this research work through the grant number MTR/2019/000525.

\end{document}